\documentclass[a4paper, reqno]{amsart}
\usepackage{amsmath, amsthm, amsfonts, amssymb}
\usepackage[applemac]{inputenc}
\usepackage[all]{xy}
\usepackage[english]{babel}
\usepackage[lite, initials]{amsrefs}
\newcommand{\forevery}[2]{\null\hfilneg\llap{$\forall#1\quad\qquad$}\hfil#2}
\let\de=\partial
\newcommand{\C}{\mathbb{C}}
\newcommand{\R}{\mathbb{R}}

\newcommand{\D}{\mathbb{D}}

\newcommand{\hatC}{\widehat{\C}}
\newcommand{\Hol}{\mathop{\mathrm{Hol}}\nolimits}
\newcommand{\Aut}{\mathop{\mathrm{Aut}}\nolimits}

\renewcommand{\Re}{\mathop{\mathrm{Re}}\nolimits}
\renewcommand{\Im}{\mathop{\mathrm{Im}}\nolimits}
\newcommand{\Klim}{\mathop{\hbox{$K$-$\lim$}}\limits}
\newcommand{\cancel}[2]{\ooalign{$\hfil#1/\hfil$\crcr$#1#2$}}
\newcommand{\void}{\mathord{\mathpalette\cancel{\mathrel{\hbox{\kern0pt\raise0.8pt\hbox
	{$\scriptstyle\bigcirc$}}}}}}
\newtheorem{Theorem}{Theorem}[section]
\newtheorem{Corollary}[Theorem]{Corollary}
\newtheorem{Proposition}[Theorem]{Proposition}
\newtheorem{Lemma}[Theorem]{Lemma}
\theoremstyle{definition}
\newtheorem{Definition}[Theorem]{Definition}
\theoremstyle{remark}
\newtheorem{Remark}[Theorem]{Remark}
\theoremstyle{remark}

\title{Multipoint Julia theorems}
\author[Marco Abate]{Marco Abate$\dagger$}
\address{Marco Abate\\ Dipartimento di Matematica\\ Universit\`a di Pisa\\ Largo Pontecorvo 5, I-56127 Pisa\\ Italy.} \email{marco.abate@unipi.it}

\thanks{$\dagger$ Partially supported by 2017 PRIN grant ``Real and Complex Manifolds: Topology, Geometry and Holomorphic Dynamics", Ministry of University and Research, Italy, and by 2020 PRA grant ``Sistemi dinamici in logica, geometria, fisica matematica e scienza delle costruzioni", University of Pisa, Italy.}

\begin{document}

\dedicatory{To Edoardo Vesentini}

\begin{abstract}
Following ideas introduced by Beardon-Minda and by Baribeau-Rivard-Wegert in the context of the Schwarz-Pick lemma, we use the iterated hyperbolic difference quotients to prove a multipoint Julia lemma. As applications, we give a sharp estimate from below of the angular derivative at a boundary point, generalizing results due to Osserman, Mercer and others; and we prove a generalization to multiple fixed points of an interesting estimate due to Cowen and Pommerenke. These applications show that iterated hyperbolic difference quotients and multipoint Julia lemmas can be useful tools for exploring in a systematic way the influence of higher order derivatives on the boundary behaviour of holomorphic self-maps of the unit disk.
\end{abstract}

\maketitle

\section{Introduction}
\label{sec:0}

The classical Schwarz-Pick lemma \cites{Schwarz, Picka, Pickb, Caratheodory} says that every holomorphic self-map of the unit disk $\D\subset\C$ is a weak contraction for the Poincar\'e distance~$\omega$. More precisely, for every $f\in\Hol(\D,\D)$ and $z$, $w\in\D$ we have
\[
\omega\bigl(f(z),f(w)\bigr)\le\omega(z,w)\;,
\]
with equality for some $z_0\ne w_0$ if and only if there is equality everywhere if and only if $f$ is an automorphism of~$\D$. In particular, if $f\notin\Aut(\D)$ then we have
\begin{equation}
\omega\bigl(f(z),f(w)\bigr)<\omega(z,w)
\label{eq:0.SP}
\end{equation}
for all $z\ne w$. %Here $\omega\colon\D\times\D\to\R^+$ is the \emph{Poincar\'e distance} given by
%\[
%\omega(z,w)=\frac{1}{2}\log\frac{1+\left|\frac{z-w}{1-\overline{w}z}\right|}{1-\left|\frac{z-w}{1-\overline{w}z}\right|}\;.
%\]

In the century following the appearance of this result many improvements of \eqref{eq:0.SP} for non automorphisms have appeared; see, e.g., \cites{Dieudonne, Rogosinski, Goluzin, Mercer1997, Mercer1999, Kaptanoglu, Beardon, BeardonCarne}. Surprisingly, in 2004 Beardon and Minda \cite{BeardonMinda} found an elegant unified way to recover all these results, and more. 

Their idea is based on the \emph{hyperbolic difference quotient}~$f^*\colon\D\times\D\to\C$ associated to a holomorphic self-map $f\in\Hol(\D,\D)$, which is defined as follows:
\[
f^*(z,w)=
\begin{cases}
\frac{f(z)-f(w)}{1-\overline{f(w)}f(z)}\bigg/\frac{z-w}{1-\overline{w}z}&\mathrm{if}\ z\ne w\;;\\
f'(z)\frac{1-|z|^2}{1-|f(z)|^2}&\mathrm{if}\ z=w\;.
\end{cases}
\]
It is clear that for every $w\in\D$ the map $z\mapsto f^*(z,w)$ is holomorphic. Beardon and Minda observed that \eqref{eq:0.SP} is equivalent to saying that, if $f$ is not an automorphism, then $f^*(\cdot,w)$ is a holomorphic self-map of~$\D$ for every~$w\in\D$. But then one can apply the Schwarz-Pick lemma to~$f^*(\cdot,w)$, obtaining the following 3-point Schwarz-Pick lemma:

\begin{Theorem}[Beardon-Minda, 2004]
\label{th:0.tpSP}
Let $f\in\Hol(\D,\D)\setminus\Aut(\D)$. Then
\[
\omega\bigl(f^*(z,v),f^*(w,v)\bigr)\le \omega(z,w)
\]
for all $z$, $v$, $w\in\D$. Furthermore equality holds for some $z_0\ne w_0$ and $v_0$ if and only if it holds everywhere if and only if $f$ is a Blaschke product of degree~2.
\end{Theorem}

Here a \emph{Blaschke product} of \emph{degree} $d\ge 1$ is a holomorphic self-map of~$\D$ of the form
\[
B(z)=e^{i\theta}\prod_{j=1}^d \frac{z-a_j}{1-\overline{a_j}z}\;,
\]
where $\theta\in\R$ and $a_1,\ldots,a_d\in\D$. In particular, the Blaschke products of degree~1 are exactly the automorphisms of~$\D$.

As mentioned above, Beardon and Minda showed how the apparently innocuous Theorem~\ref{th:0.tpSP} can be used to recover many inequalities improving the original Schwarz-Pick lemma; we refer to their beautiful paper \cite{BeardonMinda} for details.

A consequence of Theorem~\ref{th:0.tpSP} is that if $f$ is not a Blaschke product of degree~2 then $f^*(\cdot,w)$ is not an automorphism of~$\D$; therefore its hyperbolic difference quotient again is a holomorphic self-map of~$\D$, and hence we can apply the classical Schwarz-Pick lemma to get a 4-point Schwarz-Pick lemma --- and then, iterating the procedure, a $n$-point Schwarz-Pick lemma for any $n\ge 2$. 

This idea has been explored by Baribeau, Rivard and Wegert \cite{BaribeauRivardWegert} in the context of the Nevanlinna-Pick interpolation problem, and by Cho, Kim and Sugawa \cite{ChoKimSugawa} in more generality; see also \cites{Rivard1, Rivard2}. To state their results, we need some notations. Given $f\in\Hol(\D,\D)$ and $z$,~$w_1,\ldots,w_k\in\D$ we define the \emph{iterated hyperbolic difference quotient}~$\Delta_{w_k,\ldots,w_1}f(z)$ by setting $\Delta_{w_1}f(z)=f^*(z,w_1)$ and $\Delta_{w_k,\ldots,w_1}f(z)=\Delta_{w_k}(\Delta_{w_{k-1},\ldots,w_1}f)(z)$. Then we have a multi-point Schwarz-Pick lemma:

\begin{Theorem}[Baribeau-Rivard-Wegert, 2009]
\label{th:0.mpSP}
Given $k\ge 1$, take $f\in\Hol(\D,\D)$ not a Blaschke product of degree at most~$k$. Then
\[
\omega\bigl(\Delta_{w_k,\ldots,w_1}f(z),\Delta_{w_k,\ldots,w_1}f(w))\bigr)\le \omega(z,w)
\]
for all $z$, $w$, $w_1,\ldots,w_k\in\D$. Furthermore equality holds for some $z_0\ne w_0$ and $w_1,\ldots,w_k$ if and only if it holds for all $z$, $w$, $w_1,\ldots,w_k\in\D$ if and only if $f$ is a Blaschke product of degree~$k+1$.
\end{Theorem}

Another way of expressing the Schwarz-Pick lemma consists in saying that holomorphic self-maps of the unit disk send disks for the Poincar\'e distance into disks for the Poincar\'e distance. Julia \cite{Julia} in 1920 noticed that by moving the centers of these disks toward the boundary one can get a boundary version of the Schwarz-Pick lemma, nowadays known as Julia lemma:

\begin{Theorem}[Julia, 1920]
\label{th:0.cinque} 
Let $f\in\Hol(\D,\D)$ and $\sigma\in\de\D$ be such that 
\[
\beta_f(\sigma):=\liminf_{z\to\sigma}\frac{1-|f(z)|}{1-|z|}=\alpha<\infty\;.
\]
Then there exists a unique $\tau\in\de\D$ such that 
\begin{equation}
\frac{|\tau-f(z)|^2}{1-|f(z)|^2}\le\beta_f(\sigma)\frac{|\sigma-z|^2}{1-|z|^2}
\label{eq:0.J}
\end{equation}
for every $z\in\D$.
Moreover, equality in \eqref{eq:0.J} holds at one point if and only if it holds everywhere if and only if 
$f\in\Aut(\D)$.
\end{Theorem}

The number $\beta_f(\sigma)\in(0,+\infty]$ is the \emph{boundary dilation coefficient} of~$f$ at~$\sigma$, and it is the absolute value of the \emph{angular derivative}~$f'(\sigma)$, the non-tangential limit  of~$f'$ at~$\sigma$, which is known to exist thanks to the Julia-Wolff-Carath\'eodory theorem. It is well-known that $\beta_f(\sigma)>0$ always; furthermore, if $\beta_f(\sigma)<+\infty$ then the point~$\tau$ appearing in the statement of Julia lemma is the non-tangential limit of~$f$ at~$\sigma$, that we will denote by~$f(\sigma)$. 

The geometrical meaning of \eqref{eq:0.J} is that if $\beta_f(\sigma)<+\infty$ then $f$ sends horocycles centered at~$\sigma$ in horocycles centered at~$f(\sigma)$, where 
a \emph{horocycle}~$E(\sigma,R)$ of \emph{center}~$\sigma\in\de\D$ and \emph{radius} $R>0$ is given by
\[
E(\sigma,R)=\left\{z\in\D\biggm|\frac{|\sigma-z|^2}{1-|z|^2}<R\right\}\;.
\]
Geometrically, $E(\sigma,R)$ is an Euclidean disk or radius $R/(R+1)$ internally tangent at~$\de\D$ in~$\sigma$.

The aim of this paper is to obtain a multipoint version of Julia lemma along the lines of Theorems~\ref{th:0.tpSP} and~\ref{th:0.mpSP}. The paper~\cite{BeardonMinda} contains a 3-point Julia lemma, but its statement does not involve the hyperbolic difference quotient, and it is in a slightly different spirit. Closer to our aims is \cite{Mercer2000}*{Proposition~4.1}; but its (Euclidean) statement is quite involved and not easy to use (see Remark~\ref{rem:1.Mercer}). 

Our idea then is to obtain a version of Theorem~\ref{th:0.cinque} involving the iterated hyperbolic difference quotients. The main difference between the Schwarz-Pick lemma and the Julia lemma is that the latter works only for maps with finite boundary dilation coefficient. So the main result allowing our approach to start is the following (see Proposition~\ref{th:1.beta}):

\begin{Proposition}
\label{th:0.main}
Let $f\in\Hol(\D,\D)$ and $\sigma\in\de\D$ be such that $\beta_f(\sigma)<+\infty$. Then
\[
\beta_{\Delta_wf}(\sigma)=\beta_f(\sigma)\frac{1-|f(w)|^2}{|f(\sigma)-f(w)|^2}-\frac{1-|w|^2}{|\sigma-w|^2}
\]
for all $w\in\D$. In particular $\beta_{\Delta_{w_0}f}(\sigma)<+\infty$ for some $w_0\in\D$ if and only if $\beta_{\Delta_{w}f}(\sigma)<+\infty$ for all $w\in\D$ if and only if $\beta_f(\sigma)<+\infty$. 
\end{Proposition}

So if the boundary dilation coefficient is finite for $f$ it remains finite for all the iterated hyperbolic difference quotients of~$f$. This allows us to obtain a multi-point Julia lemma (see Theorem~\ref{th:1.mpJulia}):

\begin{Theorem}
\label{th:0.mpJulia}
Given $k\ge 1$, take $f\in\Hol(\D,\D)$ not a Blaschke product of degree at most~$k$. Let $\sigma\in\de\D$ be such that $\beta_f(\sigma)<+\infty$.
Then
\begin{equation}
\frac{|\Delta_{w_k,\ldots,w_1}f(\sigma)-\Delta_{w_k,\ldots,w_1}f(z)|^2}{1-|\Delta_{w_k,\ldots,w_1}f(z)|^2}\le\beta_{\Delta_{w_k,\ldots,w_1}f}(\sigma)\frac{|\sigma-z|^2}{1-|z|^2}
\label{eq:0.mpJulia}
\end{equation}
for all $z$, $w_1,\ldots,w_k\in\D$. Moreover, equality occurs for some $z_0$,~$w_1,\ldots,w_k\in\D$ if and only if it occurs everywhere if and only if $f$ is a Blaschke product of degree~$k+1$.
\end{Theorem}

We now describe two applications of this theorem. We mentioned before that the boundary dilation coefficient is always strictly positive. In some instances it is useful to have a more explicit bound from below, like the classical one
\begin{equation}
\beta_f(\sigma)\ge\frac{1-|f(0)|}{1+|f(0)|}\;.
\label{eq:0.CP}
\end{equation}
In Section~\ref{sec:3} we shall prove a much more precise estimate (see Theorem~\ref{th:2.bassogen}):

\begin{Theorem}
\label{th:0.bassogen}
Given $k\ge 0$ let $f\in\Hol(\D,\D)$ be not a Blaschke product of degree at most~$k$. Take $\sigma\in\de\D$ with $\beta_f(\sigma)<+\infty$. Then
\begin{equation}
\beta_f(\sigma)\ge\sum_{j=0}^k\frac{1-|w_{j+1}|^2}{|\sigma-w_{j+1}|^2}\prod_{h=0}^j\frac{|\Delta_{w_h,\ldots,w_1}f(\sigma)-\Delta_{w_h,\ldots,w_1}f(w_{h+1})|^2}{1-|\Delta_{w_h,\ldots,w_1}f(w_{h+1})|^2}
\label{eq:0.bassogen}
\end{equation}
for every $w_1,\ldots,w_{k+1}\in\D$, where $\Delta_{w_h,\ldots,w_1}f=f$ when $h=0$. Furthermore we have equality in \eqref{eq:0.bassogen} for some $w_1,\ldots,w_{k+1}\in\D$ if and only if we have equality for all $w_1,\ldots,w_{k+1}\in\D$ if and only if $f$ is a Blaschke product of degree~$k+1$.
\end{Theorem}

In particular we have the following corollary (see Corollary~\ref{th:2.bassot}):

\begin{Corollary}
\label{th:0.bassot}
Given $k\ge 0$ let $f\in\Hol(\D,\D)$ be not a Blaschke product of degree at most~$k$. Take $\sigma\in\de\D$ with $\beta_f(\sigma)<+\infty$. Then
\[
\beta_f(\sigma)\ge\sum_{j=0}^k\prod_{h=0}^j\frac{1-|\Delta_{\mathbf{O}_h}f(0)|}{1+|\Delta_{\mathbf{O}_h}f(0)|}\;,
\]
where $\mathbf{O}_h=(0,\ldots,0)\in\D^h$ is the origin of~$\C^h$, and $\Delta_{\mathbf{O}_h}f=f$ when $h=0$. Furthermore we have equality if and only if $f$ is a Blaschke product of degree~$k+1$ and $\Delta_{\mathbf{O}_h}f(0)=|\Delta_{\mathbf{O}_h}f(0)|\Delta_{\mathbf{O}_h}f(\sigma)$
for all $h=0,\ldots,k$.
\end{Corollary}

These results when $k=0$ recover \eqref{eq:0.CP} and when $k\ge 1$ improve previous estimates due to \"Unkelbach~\cites{Unkelbach1, Unkelbach2}, Komatu~\cite{Komatu}, Frovlova et al.~\cite{Frovlova}, Osserman~\cite{Osserman} and
Mercer~\cite{Mercer2018}.

When $f(0)=0$ the estimate \eqref{eq:0.CP} implies that $\beta_f(\sigma)\ge 1$. In 1982, Cowen and Pommerenke \cite{CowenPommerenke} proved that if moreover $f(\sigma)=\sigma$ this estimate can be improved to
\[
\beta_f(\sigma)\ge 1+\frac{|1-f'(0)|^2}{1-|f'(0)|^2}\;.
\]
More precisely, they obtained a sharp estimate valid when $f$ has a fixed point inside and $n$ fixed points on the boundary:

\begin{Theorem}[Cowen-Pommerenke, 1982]
\label{th:0.CPtrue}
Let $f\in\Hol(\D,\D)\setminus\Aut(\D)$ be such that $f(z_0)=z_0$ for some $z_0\in\D$. Assume there exist $\sigma_1,\ldots,\sigma_n\in\de\D$ distinct points with $\beta_f(\sigma_j)<+\infty$ and $f(\sigma_j)=\sigma_j$ for $j=1,\ldots,n$. Then
\begin{equation}
\sum_{j=1}^n \frac{1}{\beta_f(\sigma_j)-1}\le\frac{1-|f'(z_0)|^2}{|1-f'(z_0)|^2}\;.
\label{eq:0.CPtrue}
\end{equation}
Furthermore, equality holds if and only if $f$ is a Blaschke product of degree~$n+1$.
\end{Theorem}

In Section~\ref{sec:2} we shall show (see Proposition~\ref{th:1.CPtrue}) how to obtain this result as a consequence of our Theorem~\ref{th:0.mpJulia} for $k=1$.
More interestingly, in Section~\ref{sec:3} we shall generalize the estimate \eqref{eq:0.CPtrue} to the case of multiple fixed points (see Theorem~\ref{th:2.CPgenz}):

\begin{Theorem}
\label{th:0.CPgenz}
Let $f\in\Hol(\D,\D)$. Given $k\ge 1$, assume that $f$ is not a Blaschke product of degree at most~$k\ge 1$ and that there exists $z_0\in\D$ 
such that $f(z_0)=z_0$ and $f'(z_0)=\ldots=f^{(k-1)}(z_0)=0$. 
Take $\sigma_1,\ldots,\sigma_n\in\de\D$ distinct points such that $\beta_f(\sigma_j)<+\infty$ and 
\[
%f(\sigma_j)=\sigma_j^k\left(\frac{1-\overline{\sigma_j}z_0}{1-\sigma_j\overline{z_0}}\right)^{\!k}\,\frac{1+\left(\frac{1-\overline{z_0}\sigma_j}{\sigma_j-z_0}\right)^k z_0}{1+\left(\frac{1-z_0\overline{\sigma_j}}{\overline{\sigma_j}-\overline{z_0}}\right)^k \overline{z_0}}
f(\sigma_j)=\frac{\left(\frac{\sigma_j-z_0}{1-\overline{z_0}\sigma_j}\right)^k+z_0}{1+\overline{z_0}\left(\frac{\sigma_j-z_0}{1-\overline{z_0}\sigma_j}\right)^k}\;,
\]
for $j=1,\ldots,n$. Then
\[
\sum_{j=1}^{n}\frac{1}{\left(1+2\Re\frac{\left(f(\sigma_j)-\sigma_j\right)\overline{z_0}}{|f(\sigma_j)-z_0|^2}\right)\beta_f(\sigma_j)-k}\le\frac{1-\left|\frac{f^{(k)}(z_0)}{k!}(1-|z_0|^2)^{k-1}\right|^2}{\left|1-\frac{f^{(k)}(z_0)}{k!}(1-|z_0|^2)^{k-1}\right|^2}\;,
\]
with equality if and only if $f$ is a Blaschke product of degree~$n+k$.
\end{Theorem}

The proof in the general case is a bit delicate and requires the full force of our multipoint Julia lemma. However,
the case $z_0=0$ has a simpler statement, and actually a much easier proof (see Corollary~\ref{th:2.CPgen}):

\begin{Corollary}
\label{th:0.CPgen}
Let $f\in\Hol(\D,\D)$. Given $k\ge 1$, assume that $f$ is not a Blaschke product of degree at most~$k$ and that 
\[
f(0)=\cdots=f^{(k-1)}(0)=0\;.
\] 
Take $\sigma_1,\ldots,\sigma_n\in\de\D$ distinct points such that $\beta_f(\sigma_j)<+\infty$ and $f(\sigma_j)=\sigma_j^{k}$
for $j=1,\ldots,n$. Then
\[
\sum_{j=1}^{n}\frac{1}{\beta_f(\sigma_j)-k}\le\frac{1-\left|\frac{f^{(k)}(0)}{k!}\right|^2}{\left|1-\frac{f^{(k)}(0)}{k!}\right|^2}\;,
\]
with equality if and only if $f$ is a Blaschke product of degree~$n+k$.
\end{Corollary}

Summing up, these applications show that iterated hyperbolic difference quotients and multipoint Julia lemmas can be an useful tool for exploring in a systematic way the influence of higher order derivatives on the boundary behaviour of holomorphic self-maps of the unit disk.

This paper is organized as follows. In Section~\ref{sec:1} we shall collect a number of preliminary definitions and results that we shall need later on. In Section~\ref{sec:2} we shall discuss 2-point Julia lemmas, proving in particular Proposition~\ref{th:0.main}. Finally, in Section~\ref{sec:3} we shall introduce our general multipoint Julia lemma and its applications, proving in particular Theorems~\ref{th:0.mpJulia},~\ref{th:0.bassogen},~\ref{th:0.CPgenz} and Corollaries~\ref{th:0.bassot} and~\ref{th:0.CPgen}.

\medbreak

\noindent\emph{Ackowledgments.} This paper is respectfully dedicated to the memory of my advisor, Edoardo Vesentini, who, among (many) other things, showed me how beautiful and elegant complex analysis can be, in one, several and infinitely many variables. 

\section{Preliminaries}
\label{sec:1}

In this section we collect a few known results that shall be useful later on.

\subsection{Blaschke products}

\begin{Definition}
\label{def:I.2.Blaschke}
A (finite) \emph{Blaschke product} $B\in\Hol(\D,\D)$ is a holomorphic self-map of~$\D$ continuous up to the boundary
with $B(\de\D)\subseteq\de\D$. Since a Blaschke product $B$ cannot vanish in a neighbourhood of~$\de\D$ it must have a finite number $d\ge 0$ of zeroes in~$\D$, counted with respect to their multiplicity. The number $d$ is the \emph{degree} of~$B$.  We shall denote by $\mathcal{B}_d$ the set of Blaschke products of degree~$d\ge 1$, and by $\mathcal{B}_0$ the set of constant functions of modulus~1.
\end{Definition}

\begin{Lemma}
\label{th:I.2.Blaschke}
A function $B\in\Hol(\D,\D)$ is a Blaschke product of degree $d\ge 0$ if and only if there are $\theta\in\R$ and $a_1,\ldots, a_d\in\D$ such that
\begin{equation}
B(z)=e^{i\theta}\prod_{j=1}^d\frac{z-a_j}{1-\bar{a_j}z}\;.
\label{eq:1.B}
\end{equation}
In particular, if $\gamma\in\Aut(\D)$ then $B\circ\gamma$ and $\gamma\circ B$ are still Blaschke products of the same degree~$d$.
\end{Lemma}

\begin{proof}
Since
\[
\frac{\sigma-a}{1-\overline{a}\sigma}=\overline{\sigma}\frac{\sigma-a}{\overline{\sigma}-\overline{a}}\in\de\D
\]
for all $\sigma\in\de\D$ and $a\in\D$, it is clear that all maps of the form \eqref{eq:1.B} are Blaschke products of degree~$d$, with zeroes in~$a_1,\ldots,a_d$.

Conversely, assume that $B$ is a Blaschke product of degree~$d\ge 0$. If $d=0$ then
the maximum principle applied to $1/B$ implies that $|B|\equiv 1$, and hence $B\equiv e^{i\theta}$ for a suitable $\theta\in\R$.

Assume $d\ge 1$, let $a_1,\ldots,a_d\in\D$ be the zeroes of $B$, listed accordingly to their multiplicities, and put
\[
B_0(z)=\prod_{j=1}^d\frac{z-a_j}{1-\bar{a_j}z}\;.
\]
Then $B/B_0$ is holomorphic without zeroes in~$\D$ and $|B/B_0|=|B_0/B|\equiv 1$ on~$\de D$. By the maximum principle we get $|B/B_0|$, $|B_0/B|\le 1$,
and thus $B/B_0$ is a constant of modulus~1, as required.

If $\gamma\in\Aut(\D)$ then $B\circ\gamma$ obviously is a Blaschke product of the same degree. On the other hand, clearly $\gamma\circ B$ is still a Blaschke product. Moreover, if $a=\gamma^{-1}(0)$ then $(\gamma\circ B)(z)=0$ if and only if $B(z)=a$ if and only if
\[
e^{i\theta}\prod_{j=1}^d (z-a_j)=a\prod_{j=1}^d(1-\bar{a_j}z)\;.
\]
This is a polynomial equation of degree exactly $d$; thus $\gamma\circ B$ has exactly $d$ zeroes, counted with multiplicities, and we are done.
\end{proof}

In particular, the Blaschke products of degree~1 are exactly the automorphisms of~$\D$, that is $\mathcal{B}_1=\Aut(\D)$. If $w\in\D$ we shall denote by $\gamma_w$ the automorphism
\[
\gamma_w(z)=\frac{z-w}{1-\bar{w}z}\;.
\]

Later on we shall need the following 

\begin{Lemma}
\label{th:1.Bz}
Let $\sigma_1,\ldots,\sigma_n\in\de\D$ be distinct points, $a_1,\ldots,a_n\in\R^+$ and $B\in\mathcal{B}_d$ for some $n\ge 1$ and~$d\ge 0$. If $B\not\equiv 1$ define $h\colon\D\to\C$ by 
\[
\frac{1+h(z)}{1-h(z)}=\sum_{j=1}^n a_j\frac{\sigma_j+z}{\sigma_j-z}+\frac{1+B(z)}{1-B(z)}\;.
\]
Then $h\in\mathcal{B}_{n+d}$.
\end{Lemma}

\begin{proof}
A quick computation yields
\[
h(z)=\frac{2B(z)+\bigl(1-B(z)\bigr)S(z)}{2+\bigl(1-B(z)\bigr)S(z)}\;,
\]
where 
\[
S(z)=\sum_{j=1}^n a_j\frac{\sigma_j+z}{\sigma_j-z}\;;
\]
in particular $h$ is a rational function of degree~$n+d$ because numerator and denominator have no common factors.

When $z\in\D$ we have
\[
\Re\frac{1+h(z)}{1-h(z)}=\sum_{j=1}^n a_j\Re\frac{\sigma_j+z}{\sigma_j-z}+\Re\frac{1+B(z)}{1-B(z)}=\sum_{j=1}^n a_j\frac{1-|z|^2}{|\sigma_j-z|^2}+\frac{1-|B(z)|^2}{|1-B(z)|^2}>0\;;
\]
this yields $h(z)\in\D$, and hence $h(\D)\subseteq\D$. 

If $\sigma\in\de\D$ is different from $\sigma_j$ we have
\[
\frac{\sigma_j+\sigma}{\sigma_j-\sigma}=2i\frac{\Im(\sigma\overline{\sigma_j})}{|\sigma_j-\sigma|^2}\in i\R\;;
\]
thus if $\sigma\ne\sigma_1,\ldots\sigma_n$ we have $S(\sigma)=ia\in i\R$ and hence setting $B(\sigma)=e^{i\phi}\in\de\D$ we have
\[
h(\sigma)=\frac{2e^{i\phi}+(1-e^{i\phi})ia}{2+(1-e^{i\phi})ia}\in\de\D\;.
\]
To deal with $\sigma=\sigma_j$ we write
\[
h(z)=\frac{a_j(\sigma_j+z)\bigl(1-B(z)\bigr)+(\sigma_j-z)\bigl(S_j(z)+2B(z)\bigr)}{a_j(\sigma_j+z)\bigl(1-B(z)\bigr)+(\sigma_j-z)\bigl(S_j(z)+2\bigr)}
\]
where $S_j(\sigma_j)=0$; therefore we get $h(\sigma_j)=1$, also when $B(\sigma_j)=1$. 

Summing up, we have proved that $h$ is a Blaschke product; being a rational function of degree~$n+d$ we get $h\in\mathcal{B}_{n+d}$, and we are done.
\end{proof}

\subsection{The hyperbolic difference quotient}

\begin{Definition}
Let $f\in\Hol(\D,\D)$ be a holomorphic self-map of the unit disk. The \emph{hyperbolic derivative} $f^h\colon\D\to\C$ of $f$ is given by
\[
f^h(z)=\frac{f'(z)}{1-|\,f(z)|^2}\!\bigg/\!\frac{1}{1-|z|^2}=f'(z)\frac{1-|z|^2}{1-|\,f(z)|^2}\;.
\]
The \emph{hyperbolic difference quotient} $f^*\colon\D\times\D\to\C$ is given by
\[
f^*(z,w)=
\begin{cases}
\frac{f(z)-f(w)}{1-\overline{f(w)}f(z)}\bigg/\frac{z-w}{1-\overline{w}z}&\mathrm{if}\ z\ne w\;;\\
f^h(z)&\mathrm{if}\ z=w\;.
\end{cases}
\]
\end{Definition}

It is easy to check that for every $w\in\D$ the function $z\mapsto f^*(z,w)$ is holomorphic. Furthermore, the Schwarz-Pick lemma implies that $|f^*(z,w)|\le 1$ always,
and that there exists $(z_0,w_0)\in\D\times\D$ such that $|f^*(z_0,w_0)|=1$ if and only if $|f^*|\equiv 1$ if and only if  $f\in\Aut(\D)$. In particular, if $\gamma\in\Aut(\D)$ is given by
\[
\gamma(z)=e^{i\theta}\frac{z-a}{1-\bar{a}z}
\]
then it is easy to check that 
\[
\gamma^*(z,w)= e^{i\theta}\frac{1-\bar{a}w}{1-a\bar{w}}\;.
\]
This can be seen as a particular case of the following result:

\begin{Proposition}
\label{th:1.Blhdq}
Let $f\in\Hol(\D,\D)$ and $d\ge 1$. Then $f\in\mathcal{B}_d$ if and only if $f^*(\cdot,w)\in\mathcal{B}_{d-1}$ for all $w\in\D$ if and only if $f^*(\cdot,w_0)\in\mathcal{B}_{d-1}$ for some $w_0\in\D$.
\end{Proposition}

\begin{proof}
By definition we have
\begin{equation}
\gamma_{w}(z) f^*(z,w)=\gamma_{f(w)}\bigl(f(z)\bigr)
\label{eq:I.2.BB}
\end{equation}
for all $z$, $w\in\D$.
If $f^*(\cdot,w_0)\in\mathcal{B}_{d-1}$ for some $w_0\in\D$ then $B=\gamma_{w_0}f^*(\cdot,w_0)$ is a Blaschke product of degree~$d$, and thus Lemma~\ref{th:I.2.Blaschke} implies that $f=\gamma^{-1}_{f(w_0)}\circ B$ is a Blaschke product of degree~$d$.

Conversely, if $f$ is a Blaschke product of degree~$d$, then for any $w\in\D$ we have that $\gamma_{f(w)}\circ f$ is a Blaschke product of degree~$d$ vanishing at~$w$.
Therefore $\gamma_w$ is a factor of $\gamma_{f(w)}\circ f$, and \eqref{eq:I.2.BB} implies that $f^*(\cdot,w)$ is a Blaschke product of degree~$d-1$.
\end{proof}

%\begin{Proposition}[Beardon-Minda, 2004]
%\label{th:I.2.tpSP}
%Let $f\in\Hol(\D,\D)\setminus\Aut(\D)$. Then
%\begin{equation}
%\omega\bigl(f^*(z,v),f^*(w,v)\bigr)\le \omega(z,w)
%\label{eq:I.2.tpSP}
%\end{equation}
%for all $z$, $v$, $w\in\D$. Furthermore equality holds for some $z_0\ne w_0$ and $v_0$ if and only if it holds everywhere if and only if $f$ is a Blaschke product of degree~2.
%\end{Proposition}
%
%\begin{proof}
%Given $v\in\D$, we have seen that $f^*(\cdot,v)\in\Hol(\D,\D)$ because $f\notin\Aut(\D)$, and thus \eqref{eq:I.2.tpSP} follows from the usual Schwarz-Pick lemma. 
%
%By the same argument, if we have equality in \eqref{eq:I.2.tpSP} for some $z_0\ne w_0$ and $v_0$ it follows that $f^*(\cdot,v_0)\in\Aut(\D)$. Now, we clearly have
%\begin{equation}
%\gamma_{v_0}(z) f^*(z,v_0)=\gamma_{f(v_0)}\bigl(f(z)\bigr)\;.
%\label{eq:I.2.BB}
%\end{equation}
%If $f^*(\cdot,v_0)\in\Aut(\D)$ then $B=\gamma_{v_0}f^*(\cdot,v_0)$ is a Blaschke product of degree~2; then Lemma~\ref{th:I.2.Blaschke} implies that $f=\gamma^{-1}_{f(v_0)}\circ B$ is a Blaschke product of degree~2, as claimed.
%
%Conversely, if $f$ is a Blaschke product of degree~2, then for any $v\in\D$ we have that $\gamma_{f(v)}\circ f$ is a Blaschke product of degree~2 vanishing at~$v$.
%Therefore $\gamma_v$ is a factor of $\gamma_{f(v)}\circ f$, and \eqref{eq:I.2.BB} implies that $f^*(\cdot,v)$ is a Blaschke product of degree~1, that is an automorphism of~$\D$, and then in \eqref{eq:I.2.tpSP} the equality holds everywhere.
%\end{proof}

\subsection{The classical Julia lemma}

\begin{Definition}
\label{def:I.2.1.horo}
The 
\emph{horocycle}~$E(\sigma,R)\subset\D$ of \emph{center}~$\sigma\in\partial\D$ and 
\emph{radius}~$R>0$ is given by 
\begin{equation}
E(\sigma,R)=\biggl\{z\in\D\biggm|\frac{|\sigma-z|^2}{1-|z|^2}<R\biggr\}\;.
\label{eq:I.2.1.horo}
\end{equation}
Geometrically, $E(\sigma,R)$ is 
the euclidean disk of radius~$R/(R+1)$ internally
tangent to~$\partial\D$ in~$\sigma$.
\end{Definition}

\begin{Definition}
\label{def:I.2.bdc}
Given $f\in\Hol(\D,\D)$ and $\sigma$,~$\tau\in\de\D$, set
\begin{equation}
\beta_f(\sigma,\tau)=\sup_{z\in\D}\biggl\{\frac{|\tau-f(z)|^2}{
	1-|f(z)|^2}\biggm/\frac{|\sigma-z|^2}{1-|z|^2}\biggr\}\;.
\label{eq:I.2.unouno}
\end{equation}
The \emph{boundary dilation coefficient\/} of~$f$ at~$\sigma$ is given by
\begin{equation}
\beta_f(\sigma)=\inf_{\tau\in\de\D}\beta_f(\sigma,\tau)\in[0,+\infty]\;.
\label{eq:I.2.ora}
\end{equation}
\end{Definition}

\begin{Remark}
\label{rem:I.2.bdc}
By definition 
\[
\forevery{R>0}  f\bigl(E(\sigma,R)\bigr)\subseteq E\bigl(\tau,\beta_f(\sigma,\tau)R\bigr)\;.
\]
In particular, for every~$f\in\Hol(\D,\D)$ and~$\sigma\in\de\D$ there is at most 
one point~$\tau\in\de\D$ such that $\beta_f(\sigma,\tau)$ is finite. Indeed, if we had $\beta_f(\sigma,\tau_j)<+\infty$ for two distinct points $\tau_1$,~$\tau_2\in\de\D$ 
we would get a contradiction choosing $R$ so small that 
\[
E(\tau_1,\beta R)\cap E(\tau_2,\beta R)=\void\;,
\] 
where $\beta=\max\{\beta_f(\sigma,\tau_1),\beta_f(\sigma,\tau_2)\}$.
\end{Remark}

The following well-known result gives us an alternative way to compute the boundary dilation coefficient (for a proof see, e.g., \cite{Abate}*{Proposition~1.2.6}):

\begin{Proposition} 
\label{th:I.2.sei} 
Take $f\in\Hol(\D,\D)$ and $\sigma\in\de
\D$. Then 
\[
\beta_f(\sigma)=\liminf_{z\to\sigma}\frac{1-|f(z)|}{1-|z|}\;.%=\liminf_{z\to\sigma}\frac{1-|f(z)|^2}{1-|z|^2}\;.
\]
\end{Proposition}

Furthermore (see, e.g., \cite{Abate}*{Lemma~1.2.4}):

\begin{Lemma} 
\label{th:I.2.quattro} Let $f\colon\D\to\D$ be holomorphic. Then
\begin{equation}
\forevery{z\in\D}\frac{1-|f(z)|}{1-|z|}\ge\frac{1-|f(0)|}{1+|f(0)|}>0\;.
\label{eq:I.2.bassob}
\end{equation}
In particular, for all $\sigma\in\de\D$ we have
\begin{equation}
\beta_f(\sigma)\ge \frac{1-|f(0)|}{1+|f(0)|}>0\;.
\label{eq:I.2.basso}
\end{equation}
Moreover, equality in \eqref{eq:I.2.bassob} holds at one point $z_0\ne 0$ (and hence everywhere) 
if and only if $f(z)=e^{i\theta}z$ for a suitable~$\theta\in\R$.
\end{Lemma}

We can now state the classical Julia lemma~\cite{Julia}:

\begin{Theorem}[Julia lemma]
\label{th:I.2.cinque} 
Let $f\in\Hol(\D,\D)$, and choose $\sigma\in\de\D$ so that $\beta_f(\sigma)<+\infty$. 
%\begin{equation}
%\liminf_{z\to\sigma}\frac{1-|f(z)|}{1-|z|}=\alpha<\infty\;.
%\label{eq:I.2.otto}
%\end{equation}
%Then there exists a unique $\tau\in\de\D$ such that for every $z\in\D$
Let $\tau\in\de\D$ be the unique point of~$\de\D$ such that $\beta_f(\sigma, \tau)<+\infty$. Then 
\begin{equation}
\frac{|\tau-f(z)|^2}{1-|f(z)|^2}\le\beta_f(\sigma)\frac{|\sigma-z|^2}{1-|z|^2}\;,
\label{eq:I.2.nove}
\end{equation}
that is 
\begin{equation}
\forevery{R>0}f\bigl(E(\sigma,R)\bigr)\subseteq E(\tau,\beta_f(\sigma) R)\;.
	\label{eq:I.2.anche}
\end{equation} 
Moreover, equality in \eqref{eq:I.2.nove} holds at one point (and hence everywhere) if and only if 
$f\in\Aut(\D)$.
\end{Theorem}

As a consequence we have another way for computing the boundary dilation coefficient:

\begin{Corollary}
\label{th:1.bdc2}
Take $f\in\Hol(\D,\D)$ and $\sigma\in\de
\D$. Then
\[
\beta_f(\sigma)=\lim_{r\to1^-}\frac{1-|f(r\sigma)|}{1-r}=\lim_{r\to1^-}\frac{1-|f(r\sigma)|^2}{1-r^2}\;.
\]
\end{Corollary}

\begin{proof}
By Proposition~\ref{th:I.2.sei} we have
\[
\beta_f(\sigma)\le\liminf_{r\to1^-}\frac{1-|f(r\sigma)|}{1-r}\;;
\]
in particular the first equality is proven when $\beta_f(\sigma)=+\infty$. Assume then that $\beta_f(\sigma)<+\infty$. An easy computation shows that
$r\sigma\in\de E\bigl(\sigma,\frac{1-r}{1+r}\bigr)$; therefore Theorem~\ref{th:I.2.cinque} yields $f(r\sigma)\in \overline{E\bigl(\tau,\beta_f(\sigma)\frac{1-r}{1+r}\bigr)}$
for a suitable $\tau\in\de\D$. Since $E(\tau,R)$ is an Euclidean disk of radius~$R/(R+1)$ internally tangent to~$\de\D$ in~$\tau$ it follows that
\[
1-|f(r\sigma)|\le|\tau-f(r\sigma)| \le 2\frac{\beta_f(\sigma)\frac{1-r}{1+r}}{1+\beta_f(\sigma)\frac{1-r}{1+r}}\;.
\]
Therefore
\[
\limsup_{r\to 1^-}\frac{1-|f(r\sigma)|}{1-r}\le\limsup_{r\to 1^-}\frac{2\beta_f(\sigma)}{1+r+\beta_f(\sigma)(1-r)}=\beta_f(\sigma)
\]
and the first equality is proved.

To prove the second equality, first of all notice that 
\begin{equation}
\frac{2}{1+r}\frac{1-|f(r\sigma)|}{1-r}\ge \frac{1+|f(r\sigma)|}{1+r}\frac{1-|f(r\sigma)|}{1-r}=\frac{1-|f(r\sigma)|^2}{1-r^2}\ge\frac{1}{1+r}\frac{1-|f(r\sigma)|}{1-r} \;.
\label{eq:1.bdcsquare}
\end{equation}
From this the second equality immediately follows when $\beta_f(\sigma)=+\infty$. If $\beta_f(\sigma)<+\infty$ then $|f(r\sigma)|\to 1$ as $r\to 1^-$, 
and thus the assertion follows again from \eqref{eq:1.bdcsquare}.
\end{proof}

%let us for simplicity call $\hat\beta_f(\sigma)$ the right-hand side.
%We have
%From this it clearly follows that $\beta_f(\sigma)\ge\hat\beta_f(\sigma)\ge \frac{1}{2}\beta_f(\sigma)$. 
%In particular, $\beta_f(\sigma)=+\infty$ if and only if $\hat\beta_f(\sigma)=+\infty$; so to end the proof we can assume $\beta_f(\sigma)<+\infty$. We already have $\beta_f(\sigma)\ge\hat\beta_f(\sigma)$. To get the opposite inequality, choose a sequence $\{z_\nu\}\subset\D$ converging to~$\sigma$ and such that
%\[
%\lim_{\nu\to\infty}\frac{1-|f(z_\nu)|^2}{1-|z_\nu|^2}=\hat\beta_f(\sigma)\;.
%\] 
%Since $\hat\beta_f(\sigma)<+\infty$ we necessarily have $|f(z_\nu)|\to 1$; therefore
%\[
%\beta_f(\sigma)\le\liminf_{\nu\to+\infty}\frac{1-|f(z_\nu)|}{1-|z_\nu|}=\liminf_{\nu\to+\infty}\frac{1+|z_\nu|}{1+|f(z_\nu)|}\frac{1-|f(z_\nu)|^2}{1-|z_\nu|^2}=\hat\beta_f(\sigma)\;,
%\]
%and we are done.
%
%The second equality follows immediately from \eqref{eq:1.bdcsquare} when $\beta_f(\sigma)=+\infty$. If $\beta_f(\sigma)<+\infty$ then $|f(r\sigma)|\to 1$ as $r\to 1^-$, 
%and thus the assertion follows again from \eqref{eq:1.bdcsquare}.
%\end{proof}

\begin{Definition}
\label{def:I.2.Stolz}
Given $\tau\in\partial\D$ and $M>0$, the \emph{Stolz region\/} $K(\tau,
M)$ of \emph{vertex}~$\tau$ and \emph{amplitude}~$M$ is
\begin{equation}
K(\tau,M)=\biggl\{z\in\D\biggm|\frac{|\tau-z|}{1-|z|}<M\biggl\}\;.
	\label{eq:I.2.Stolzo}
\end{equation}
Note that $K(\tau,M)=\void$ if $M\le1$, for $|\tau-z|\ge1-|z|$.
\end{Definition}

\begin{Definition}
\label{def:I.2.ntlimit}
We say that a function $f\colon\D\to\hatC$ has \emph{non-tangential} (or \emph{angular})
limit~$c\in\hatC$ at~$\sigma\in\de\D$ if $f(z)\to c$ as $z$~tends to~$\sigma$ 
within~$K(\sigma,M)$ for any~$M>1$. When this happen we shall write
\[
\Klim_{z\to\sigma}f(z)=c
\]
and denote ~$c$ by~$f(\sigma)$.
\end{Definition}

We end this preliminary section recalling the famous Julia-Wolff-Carath\'eodory theorem~\cites{Wolff, Caratheodorya}; for a proof see, e.g., \cite{Abate}*{Theorem~1.2.7}.

%\begin{Lemma}
%\label{th:I.2.ntlim}
%Let $f\in\Hol(\D,\D)$ and $\sigma\in\de\D$ be such that $\beta_f(\sigma)<+\infty$. Then  $f$ has non-tangential limit $\tau\in\de\D$ at~$\sigma$, where $\tau$ is the unique point of~$\de\D$ such that $\beta_f(\sigma)=\beta_f(\sigma,\tau)$.
%\end{Lemma}

\begin{Theorem}[Julia-Wolff-Carath\'eodory; 1926] 
\label{th:I.2.dieci} 
Let $f\in\Hol(\D,\D)$ and $\tau$,~$\sigma\in
\partial\D$. Then
\begin{equation}
\Klim_{z\to\sigma}\frac{\tau-f(z)}{\sigma-z}=\tau\bar\sigma\beta_f
	(\sigma,\tau)\;.
	\label{eq:I.2.JC}
\end{equation}
If this non-tangential limit is finite then $\beta_f(\sigma)=\beta_f(\sigma,\tau)<+\infty$, the function $f$ has non-tangential limit $\tau$ at~$\sigma$ and
\begin{equation}
\Klim_{z\to\sigma}f^\prime(z)=\tau\bar\sigma\beta_f(\sigma)\;.
	\label{eq:I.2.JCdue}
\end{equation}
\end{Theorem}

\section{2-point Julia lemma}
\label{sec:2}

As anticipated in the introduction, our 2-point Julia lemma will be obtained by applying the classical Julia lemma to the function $f^*(\cdot,w)$. To do so we need to compute the boundary dilation coefficient of the hyperbolic difference quotient; this is done in the next two results. 

\begin{Lemma}
\label{th:1.f*bound}
Let $f\in\Hol(\D,\D)$ and $\sigma\in\de\D$ be such that $f$ has non-tangential limit $f(\sigma)\in\de\D$ at~$\sigma$. Then
\[
\Klim_{z\to\sigma} f^*(z,w)=\frac{f(\sigma)-f(w)}{1-\overline{f(w)}f(\sigma)}\bigg/\frac{\sigma-w}{1-\overline{w}\sigma}%\\
=\overline{f(\sigma)}\sigma\,\frac{f(\sigma)-f(w)}{\overline{f(\sigma)}-\overline{f(w)}}\,\frac{\overline{\sigma}-\overline{w}}{\sigma-w}\in\de\D
\]
for all $w\in\D$. 
\end{Lemma}

\begin{proof}
It follows immediately from the definition of $f^*$.
\end{proof}

\begin{Definition}
\label{def:1.f*sigma}
Let $f\in\Hol(\D,\D)$ and $\sigma\in\de\D$ be such that $f$ has non-tangential limit $f(\sigma)\in\de\D$ at~$\sigma$.  Given $w\in\D$ we set
\[
f^*(\sigma,w)=\Klim_{z\to\sigma} f^*(z,w)\;;
\]
in particular, $|f^*(\sigma, \cdot)|\equiv 1$.
\end{Definition}

\begin{Remark}
\label{rem:1.f*sigma}
For the sake of completeness, we remark that 
\[
\Klim_{w\to\sigma} f^*(z,w)=\frac{f(z)-f(\sigma)}{1-\overline{f(\sigma)}f(z)}\bigg/\frac{z-\sigma}{1-\overline{\sigma}z}\equiv f(\sigma)\overline{\sigma}
\]
for all $z\in\D$. Moreover, \eqref{eq:I.2.JC} yields
\[
\Klim_{w\to\sigma} f^*(\sigma,w)=\overline{f(\sigma)}\sigma\frac{f(\sigma)\overline{\sigma}\beta_f(\sigma)}{\overline{f(\sigma)}\sigma\beta_f(\sigma)}=f(\sigma)\overline{\sigma}\;.
\]
\end{Remark} 

\begin{Proposition}
\label{th:1.beta}
Let $f\in\Hol(\D,\D)$ and $\sigma\in\de\D$ be such that $\beta_f(\sigma)<+\infty$.
%\begin{equation}
%\liminf_{z\to\sigma}\frac{1-|f(z)|}{1-|z|}=\beta_f<+\infty\;.
%\label{eq:1.beta}
%\end{equation}
Denote by $f(\sigma)\in\de\D$ the non-tangential limit of~$f$ at~$\sigma$. Then
\begin{equation}
\begin{aligned}
\liminf_{z\to\sigma}\frac{1-|f^*(z,w)|}{1-|z|}&=\beta_f(\sigma)\frac{1-|f(w)|^2}{|f(\sigma)-f(w)|^2}-\frac{1-|w|^2}{|\sigma-w|^2}\\
&=\frac{1-|f(w)|^2}{|f(\sigma)-f(w)|^2}\left[\beta_f(\sigma)-\frac{|f(\sigma)-f(w)|^2}{1-|f(w)|^2}\bigg/\frac{|\sigma-w|^2}{1-|w|^2}\right]
\label{eq:1.beta}
\end{aligned}
\end{equation}
for all $w\in\D$. Moreover, the left-hand side vanishes for some $w_0\in\D$ if and only if it vanishes for all~$w\in\D$ if and only if $f\in\Aut(\D)$.
\end{Proposition}

\begin{proof}
First of all we have 
\[
\begin{aligned}
1-|f^*(z,w)|^2&=1-\frac{\left|\frac{f(z)-f(w)}{1-\overline{f(w)}f(z)}\right|^2}{\left|\frac{z-w}{1-\overline{w}z}\right|^2}=
\frac{\left(1-\left|\frac{f(z)-f(w)}{1-\overline{f(w)}f(z)}\right|^2\right)-\left(1-\left|\frac{z-w}{1-\overline{w}z}\right|^2\right)}{\left|\frac{z-w}{1-\overline{w}z}\right|^2}\\
&=\frac{\frac{(1-|f(z)|^2)(1-|f(w)|^2)}{|1-\overline{f(w)}f(z)|^2}-\frac{(1-|z|^2)(1-|w|^2)}{|1-\overline{w}z|^2}}{\left|\frac{z-w}{1-\overline{w}z}\right|^2}\\
&=(1-|z|^2)\frac{\frac{1-|f(w)|^2}{|1-\overline{f(w)}f(z)|^2}\frac{1-|f(z)|^2}{1-|z|^2}-\frac{1-|w|^2}{|1-\overline{w}z|^2}}{\left|\frac{z-w}{1-\overline{w}z}\right|^2}
\end{aligned}
\]
Recalling Corollary~\ref{th:1.bdc2} and the fact that $f$ has non-tangential limit $f(\sigma)\in\de\D$ at~$\sigma$ we obtain
\[
\begin{aligned}
\liminf_{z\to\sigma}\frac{1-|f^*(z,w)|}{1-|z|}&=%\liminf_{z\to\sigma}\frac{1-|f^*(z,w)|^2}{1-|z|^2}=
\liminf_{r\to 1^-}\frac{1-|f^*(r\sigma,w)^2|}{1-r^2}
=\frac{\frac{1-|f(w)|^2}{|1-\overline{f(w)}f(\sigma)|^2}\beta_f(\sigma)-\frac{1-|w|^2}{|1-\overline{w}\sigma|^2}}{\left|\frac{\sigma-w}{1-\overline{w}\sigma}\right|^2}\\
&=\beta_f(\sigma)\frac{1-|f(w)|^2}{|f(\sigma)-f(w)|^2}-\frac{1-|w|^2}{|\sigma-w|^2}\;,
\end{aligned}
\]
and \eqref{eq:1.beta} is proved.

Finally, by Theorem~\ref{th:I.2.cinque} the right-hand side in \eqref{eq:1.beta} vanishes at one point if and only if it vanishes identically if and only if $f\in\Aut(\D)$,
and we are done.
\end{proof}

\begin{Definition}
\label{def:1.beta*}
Given $f\in\Hol(\D,\D)$ and $\sigma\in\de\D$ we put
\[
\beta^*_f(\sigma;w)=\beta_f(\sigma)\frac{1-|f(w)|^2}{|f(\sigma)-f(w)|^2}-\frac{1-|w|^2}{|\sigma-w|^2}\in[0,+\infty]
\]
for all $w\in\D$. In particular, $\beta^*_f(\sigma,w_0)=+\infty$ for some $w_0\in\D$ if and only if $\beta^*_f(\sigma;\cdot)\equiv+\infty$ if and only if $\beta_f(\sigma)=+\infty$,
and $\beta^*_f(\sigma,w_0)=0$ for some $w_0\in\D$ if and only if $\beta^*_f(\sigma;\cdot)\equiv0$ if and only if $f\in\Aut(\D)$.
\end{Definition}

As a first hint of how it is possible to use this kind of results we show how to improve \eqref{eq:I.2.basso}: 

\begin{Corollary}
\label{th:1.basso}
Let $f\in\Hol(\D,\D)$ and $\sigma\in\de\D$ be such that $\beta_f(\sigma)<+\infty$. Then
\[
\beta_f(\sigma)\frac{1-|f(w)|^2}{|f(\sigma)-f(w)|^2}\ge\frac{|w|-\left|\frac{f(w)-f(0)}{1-\overline{f(w)}f(0)}\right|}{|w|+\left|\frac{f(w)-f(0)}{1-\overline{f(w)}f(0)}\right|}+\frac{1-|w|^2}{|\sigma-w|^2}
\]
for all $w\in\D\setminus\{0\}$ and
\begin{equation}
\beta_f(\sigma)\ge\frac{2}{1+|f^h(0)|}\frac{|f(\sigma)-f(0)|^2}{1-|f(0)|^2}\ge\frac{2}{1+|f^h(0)|}\frac{1-|f(0)|}{1+|f(0)|}\;.
\label{eq:1.Oss}
\end{equation}
\end{Corollary}

\begin{proof}
If we apply \eqref{eq:I.2.basso} to $f^*(\cdot,w)$ we get
\[
\beta^*_f(\sigma,w)\ge\frac{1-|f^*(0,w)|}{1+|f^*(0,w)|}
\]
for all $w\in\D$. Since
\[
f^*(0,w)=
\begin{cases}
\frac{f(w)-f(0)}{w}\frac{1}{1-\overline{f(w)}f(0)}&\mathrm{if}\ w\ne 0\;,\\
f^h(0)&\mathrm{if}\ w=0\;,
\end{cases}
\]
we immediately get the first inequality for $w\ne 0$. When $w=0$ we get
\begin{equation}
\beta_f(\sigma)\frac{1-|f(0)|^2}{|f(\sigma)-f(0)|^2}\ge\frac{1-|f^h(0)|}{1+|f^h(0)|}+1=\frac{2}{1+|f^h(0)|}
\label{eq:1.Ossdue}
\end{equation}
and we are done.
\end{proof}

Notice that when $f\notin\Aut(D)$ we have $|f^h(0)|<1$ and thus \eqref{eq:1.Oss} is strictly stronger than \eqref{eq:I.2.basso}.

The inequality \eqref{eq:1.Oss} was already known (see, e.g., \cite{Osserman}); however, in the next section we shall substantially improve it (see Theorem~\ref{th:2.bassogen} and its corollaries).

We can now state and prove our 2-point Julia lemma:

\begin{Theorem}
\label{th:1.2pJulia}
Let $f\in\Hol(\D,\D)\setminus\Aut(\D)$ and $\sigma\in\de\D$ be such that $\beta_f(\sigma)<+\infty$.
Then
\begin{equation}
\frac{|f^*(\sigma,w)-f^*(z,w)|^2}{1-|f^*(z,w)|^2}\le\beta^*_f(\sigma;w)\frac{|\sigma-z|^2}{1-|z|^2}
\label{eq:1.2pJulia}
\end{equation}
for all $z$, $w\in\D$. %In particular,
%\[
%\frac{|f^*(\sigma,z)-f^h(z)|^2}{1-|f^h(z)|^2}\le\beta^*_f(\sigma;z)\frac{|\sigma-z|^2}{1-|z|^2}=\beta_f\,\frac{|\sigma-z|^2}{1-|z|^2}\bigg/\frac{|f(\sigma)-f(z)|^2}{1-|f(z)|^2}-1\;.
%\]
Moreover, equality in \eqref{eq:1.2pJulia} occurs for some $(z_0,w_0)\in\D\times\D$ if and only if it occurs everywhere if and only if $f$ is a Blaschke product of degree~2.
\end{Theorem}

\begin{proof}
The inequality \eqref{eq:1.2pJulia} follows from Theorem~\ref{th:I.2.cinque} applied to $f^*(\cdot,w)$. If we have equality in \eqref{eq:1.2pJulia}
for some $(z_0,w_0)\in\D\times\D$ again Theorem~\ref{th:I.2.cinque} implies that $f^*(\cdot,w_0)\in\Aut(\D)$, and then Proposition~\ref{th:1.Blhdq} implies that
$f\in\mathcal{B}_2$. Conversely, $f\in\mathcal{B}_2$ implies that $f^*(\cdot,w)\in\Aut(\D)$ for all $w\in\D$, and thus we have equality in \eqref{eq:1.2pJulia}
for all $z$,~$w\in\D$.
\end{proof}

\begin{Remark}
\label{rem:1.Mercer}
It turns out that \eqref{eq:1.2pJulia} implies a (not very illuminating) Euclidean statement, originally proved by Mercer~\cite{Mercer2000}, that in our notations can be expressed as follows: let $f\in\Hol(\D,\D)$ and $\sigma\in\de\D$ be such that $\beta_f(\sigma)<+\infty$. Take $w\in\D$ and set $\Lambda=\frac{1-|z|^2}{|\sigma-z|^2}$, $\hat\beta=\beta_f^*(\sigma;w)$, $\phi_{w}(z)=\frac{w-z}{1-\overline{w}z}$ and 
\[
L=\frac{1-|\overline{f(w)}\phi_{w}(z)|^2}{|1-\overline{f(w)}f^*(\sigma,w)\phi_{w}(z)|^2}\;.
\] 
Then for all $z\in\D$ we have $|f(z)-c_w(z)|<r_w(z)$, where
\[
\begin{aligned}
c_w(z)&=\overline{f(\sigma)}f^*(\sigma,w)\phi_{w}(z)\frac{1-|f(w)|^2}{\bigl(1-\overline{f(w)}f^*(\sigma,w)\phi_{w}(z)\bigr)^2}\frac{\hat\beta}{\hat\beta L+\Lambda}\\
&\quad+\overline{f(\sigma)}\phi_{f(w)}\bigl(f^*(\sigma,w)\phi_{w}(z)\bigr)\;,\\
r_w(z)&=|\phi_{w}(z)|\frac{1-|f(w)|^2}{\bigl|1-\overline{f(w)}f^*(\sigma,w)\phi_{w}(z)\bigr|^2}\frac{\hat\beta}{\hat\beta L+\Lambda}\;.
\end{aligned}
\]

This can be recovered as follows. Theorem \ref{th:1.2pJulia} says that $f^*(z,w)$ belongs to the horocycle of center $f^*(\sigma,w)$ and radius $\beta^*/\Lambda$, which is an Euclidean disk of center $\frac{\Lambda}{\beta^*+\Lambda}f^*(\sigma,w)$ and radius~$\frac{\beta^*}{\beta^*+\Lambda}$. Since $\phi_w(z)f^*(z,w)=\phi_{f(w)}\bigl(f(z)\bigr)$,
it follows that $\phi_{f(w)}\bigl(f(z)\bigr)$ belongs to the Euclidean disk~$D$ of center 
$\frac{\Lambda}{\beta^*+\Lambda}f^*(\sigma,w)\phi_w(z)$ and radius~$\frac{\beta^*}{\beta^*+\Lambda}|\phi_w(z)|$. Using the fact that $\phi_{f(w)}^{-1}=\phi_{f(w)}$ we get that $f(z)$ belongs to the
Euclidean disk $\phi_{f(w)}(D)$; computing center and radius of this latter disk we get the assertion.
\end{Remark}

Applying Theorem~\ref{th:I.2.dieci} to $f^*(\cdot,w)$ we get the next corollary:

\begin{Corollary}
\label{th:1.2pJWC}
Let $f\in\Hol(\D,\D)\setminus\Aut(\D)$ and $\sigma\in\de\D$ be such that $\beta_f(\sigma)<+\infty$.
Then
\[
\Klim_{z\to\sigma}\frac{f^*(\sigma,w)-f^*(z,w)}{\sigma-z}=f^*(\sigma,w)\overline{\sigma}\beta^*_f(\sigma;w)
\]
and
\[
\Klim_{z\to\sigma}\frac{\de f^*}{\de z}(z,w)=f^*(\sigma,w)\overline{\sigma}\beta^*_f(\sigma;w)\;.
\]
\end{Corollary}

As mentioned in the introduction, specializing Theorem~\ref{th:1.2pJulia} to the case $f(z_0)=z_0$ we can recover an estimate due to Cowen and Pommerenke. The main step is contained in the following

\begin{Corollary}
\label{th:1.CP}
Let $f\in\Hol(\D,\D)\setminus\Aut(\D)$ and $\sigma\in\de\D$ be such that $\beta_f(\sigma)<+\infty$.
Assume that $f(0)=0$. Then 
\begin{equation}
\frac{\left|\frac{f(\sigma)}{\sigma}-\frac{f(z)}{z}\right|^2}{1-\left|\frac{f(z)}{z}\right|^2}\le \bigl(\beta_f(\sigma)-1\bigr)\frac{|\sigma-z|^2}{1-|z|^2}
\label{eq:1.Jz}
\end{equation}
if $z\ne 0$ and
\begin{equation}
\frac{\left|\frac{f(\sigma)}{\sigma}-f'(0)\right|^2}{1-|f'(0)|^2}\le\beta_f(\sigma)-1\;.
\label{eq:1.Jzz}
\end{equation}
In particular, if furthermore $f(\sigma)=\sigma$ we get
\begin{equation}
\frac{|1-f'(0)|^2}{1-|f'(0)|^2}\le\beta_f(\sigma)-1\;.
\label{eq:1.Jzzz}
\end{equation}
Moreover, equality occurs in \eqref{eq:1.Jz} at some $z_0\in\D\setminus\{0\}$ or in \eqref{eq:1.Jzz} if and only if it always occurs if and only $f$ is a Blaschke product of degree~2.
\end{Corollary}

\begin{proof} 
If $f(0)=0$ then 
\begin{equation}
f^*(z,0)=
\begin{cases}
\frac{f(z)}{z}&\mathrm{if}\ z\ne 0\;,\\
f'(0)&\mathrm{if}\ z=0\;;
\end{cases}
\label{eq:1.recall}
\end{equation}
moreover, $f^*(\sigma,0)=f(\sigma)/\sigma$ and $\beta^*_f(\sigma;0)=\beta_f(\sigma)-1$. The assertions then follow from Theorem~\ref{th:1.2pJulia}.
\end{proof}

When $f(\sigma)=\sigma$ \eqref{eq:1.Jzzz} can be restated as
\[
\frac{1}{\beta_f(\sigma)-1}\le\frac{1-|f'(0)|^2}{|1-f'(0)|^2}=\Re\frac{1+f'(0)}{1-f'(0)}\;.
\]
Recalling that $f'(\sigma)=\beta_f(\sigma)$ when $f(\sigma)=\sigma$, where $f'(\sigma)$ is the non-tangential limit of~$f'$ at~$\sigma$ (see Theorem~\ref{th:I.2.dieci}), using \eqref{eq:1.Jz} we can now recover a result due to Cowen and Pommerenke~\cite{CowenPommerenke}, saying that a similar estimate still holds when there are several fixed points in the boundary:

\begin{Proposition}[Cowen-Pommerenke, 1982]
\label{th:1.CPtrue}
Let $f\in\Hol(\D,\D)\setminus\Aut(\D)$ be such that $f(z_0)=z_0$ for some $z_0\in\D$. Assume there exist $\sigma_1,\ldots,\sigma_n\in\de\D$ distinct points with $\beta_f(\sigma_j)<+\infty$ and $f(\sigma_j)=\sigma_j$ for $j=1,\ldots,n$. Then
\begin{equation}
\sum_{j=1}^n \frac{1}{\beta_f(\sigma_j)-1}\le\frac{1-|f'(z_0)|^2}{|1-f'(z_0)|^2}\;.
\label{eq:1.CPtrue}
\end{equation}
Furthermore, equality holds if and only if $f$ is a Blaschke product of degree~$n+1$.
\end{Proposition}

\begin{proof}
First of all, let $\phi_{z_0}(z)=(z_0-z)/(1-\overline{z_0}z)$. Then $\tilde f=\phi_{z_0}\circ f\circ\phi_{z_0}$ satisfies $\tilde f(0)=0$ and $\tilde f'(0)=f'(z_0)$. Moreover 
if we put $\tilde\sigma_j=\phi_{z_0}(\sigma_j)$ then we have $\tilde f(\tilde\sigma_j)=\tilde\sigma_j$ and $\beta_{\tilde f}(\tilde\sigma_j)=\beta_f(\sigma_j)$. Therefore in the proof without loss
of generality we can assume $z_0=0$. 

For $j=1,\ldots,n$ set $\beta_j=\beta_f(\sigma_j)$. 
We would like to prove, by induction on~$n$, that
\begin{equation}
\begin{aligned}
\frac{1-\left|f(z)/z\right|^2}{\left|1-f(z)/z\right|^2}&=\Re\left(\frac{1+f(z)/z}{1-f(z)/z}\right)\\
&\ge\sum_{j=1}^n\frac{1}{\beta_j-1}\Re\left(\frac{\sigma_j+z}{\sigma_j-z}\right)
=\sum_{j=1}^n\frac{1}{\beta_j-1}\frac{1-|z|^2}{|\sigma_j-z|^2}
\end{aligned}
\label{eq:1.CPu}
\end{equation}
for all $z\in\D$, with equality at one point (and hence everywhere) if and only if $f\in\mathcal{B}_{n+1}$. Clearly, when $z=0$ the expression $f(z)/z$ is replaced by $f'(0)$, and thus the theorem follows taking $z=0$.

For $n=1$ \eqref{eq:1.CPu} follows from Corollary~\ref{th:1.CP}.
Assume it is true for $n-1$. In particular we have
\begin{equation}
\Re\left[\frac{1+f(z)/z}{1-f(z)/z}-\sum_{j=1}^{n-1}\frac{1}{\beta_j-1}\frac{\sigma_j+z}{\sigma_j-z}\right]\ge 0\;,
\label{eq:1.CPz}
\end{equation}
with equality at one point (and hence everywhere) if and only if $f\in\mathcal{B}_n$. 
Therefore we can find $h\in\Hol(\D,\C)$ with $h(\D)\subset\overline{\D}$ so that
\begin{equation}
\frac{1+f(z)/z}{1-f(z)/z}-\sum_{j=1}^{n-1}\frac{1}{\beta_j-1}\frac{\sigma_j+z}{\sigma_j-z}=\frac{1+h(z)}{1-h(z)}\;.
\label{eq:1.CPmu}
\end{equation}
Notice that $h(\D)\subseteq\D$ unless in \eqref{eq:1.CPz} we have equality at one point (and hence everywhere); in that case $h\equiv e^{i\theta}$ for a suitable $\theta\in\R$.

If $h\equiv e^{i\theta}$, Lemma~\ref{th:1.Bz} shows that $f(z)=zB(z)$, where $B\in\mathcal{B}_{n-1}$. But then $f$, being rational of degree~$n$, can have at most $n$ fixed points, whereas we are assuming that it has $n+1$ fixed points, contradiction. Thus $h$ cannot be a constant, and we have the strict inequality in \eqref{eq:1.CPz}.

A quick computation shows that 
\[
h(z)=\frac{2\frac{f(z)}{z}-\left(1-\frac{f(z)}{z}\right)S(z)}{2-\left(1-\frac{f(z)}{z}\right)S(z)}\;,
\]
where
\[
S(z)=\sum_{j=1}^{n-1}\frac{1}{\beta_j-1}\frac{\sigma_j+z}{\sigma_j-z}\;,
\]
with the usual convention of replacing $f(z)/z$ by $f'(0)$ when $z=0$.

Put $g(z)=zh(z)$. Then $g\in\Hol(\D,\D)$, $g(0)=0$ and $g(\sigma_n)=\sigma_n$, because $h(\sigma_n)=1$. Furthermore we have
\[
h'(z)=\frac{1}{\left(2-\left(1-\frac{f(z)}{z}\right)S(z)\right)^2}\left[\frac{4}{z}\left(f'(z)-\frac{f(z)}{z}\right)+O\left(1-\frac{f(z)}{z}\right)\right]\;,
\]
and thus
\[
\Klim_{z\to\sigma_n}h'(z)=\overline{\sigma_n}(\beta_n-1)\;.
\]
Since $g'(z)=h(z)+zh'(z)$ we get $g'(\sigma_n)=\beta_n$. Since $h$ is not a constant we can apply Corollary~\ref{th:1.CP} to $g$
obtaining
\begin{equation}
\Re\left(\frac{1+h(z)}{1-h(z)}\right)\ge\frac{1}{\beta_n-1}\Re\left(\frac{\sigma_n+z}{\sigma_n -z}\right)
\label{eq:1.CPzz}
\end{equation}
which recalling the definition of $h$ gives exactly \eqref{eq:1.CPu}. 

If we have equality in one point in \eqref{eq:1.CPu} we must have equality in one point 
in \eqref{eq:1.CPzz}, and this happens if and only if $g$ is a Blaschke product of degree~2, again by Corollary~\ref{th:1.CP}. But this occurs if and only if $h\in\Aut(\D)$; putting this in \eqref{eq:1.CPmu}
we get that $f\in\mathcal{B}_{n+1}$ by Lemma~\ref{th:1.Bz}. 

To prove the converse, assume that $f\in\mathcal{B}_{n+1}$ with $f(0)=0$. Then $f(z)=zB(z)$, where $B\in\mathcal{B}_n$, and $\sigma_1,\ldots,\sigma_n\in\de\D$ are the $n$ distinct solutions of $B(z)=1$.  Let $F\colon\C\to\widehat\C$ be defined by
\[
F(z)=\frac{1+B(z)}{1-B(z)}-\sum_{j=1}^{n}\frac{1}{\beta_j-1}\frac{\sigma_j+z}{\sigma_j-z}\;.
\]
Then $\Re F|_{\de\D}\equiv 0$; this implies that $\Re F(0)=0$, which is exactly
\[
\frac{1-|f'(0)|^2}{|1-f'(0)|^2}-\sum_{j=1}^{n}\frac{1}{\beta_j-1}=0\;,
\]
and we are done.
\end{proof}

\section{Multipoint Julia lemmas}
\label{sec:3}

Our 2-point Julia lemma has been obtained by applying the classical Julia lemma to the hyperbolic difference quotient $f^*(\cdot,w)$, which is a holomorphic self-map of~$\D$ as soon as $f$ is not an automorphism of~$\D$. But if we also assume that $f$ is not a Blaschke product of degree~2 then by Proposition~\ref{th:1.Blhdq} $f^*(\cdot,w)$ is not
an automorphism of~$\D$, and so \emph{its} hyperbolic difference quotient is a holomorphic self-map of~$\D$ to which we may apply the classical Julia lemma, obtaining a 3-point Julia lemma. 

Clearly this procedure can be iterated; to do so let us introduce some notations.

\begin{Definition}
\label{def:2.mhdq}
Given $k\ge 1$ and $w_1,\ldots,w_k\in\D$ the \emph{hyperbolic $k$-th difference quotient} $\Delta_{w_k,\ldots,w_1}f$ of $f\in\Hol(\D,\D)$ is defined by induction by setting $\Delta_{w_1}f(z)=f^*(z,w_1)$ and 
\[
\Delta_{w_k,\ldots,w_1}f(z)=\Delta_{w_k}(\Delta_{w_{k-1},\ldots,w_1}f)(z)
\]
for $k\ge 2$. 
\end{Definition}

Proposition~\ref{th:1.Blhdq} ensures that $\Delta_{w_k,\ldots,w_1}f\in\Hol(\D,\D)$ as soon as $f$ is not a Blaschke product of degree at most~$k$. Moreover, if $\sigma\in\de\D$ is such that $\beta_f(\sigma)<+\infty$ by applying repeatedly Proposition~\ref{th:1.beta} we see that $\beta_{\Delta_{w_k,\ldots,w_1}f}(\sigma)$ is finite. More precisely, $\beta_{\Delta_{w_k,\ldots,w_1}f}(\sigma)$ can be recursively computed by 
\[
\begin{aligned}
\beta_{\Delta_{w_1,\ldots,w_k}f}&(\sigma)\\
&\!\!\!=\beta_{\Delta_{w_{k-1},\ldots,w_1}f}(\sigma)\frac{1-|\Delta_{w_{k-1},\ldots,w_1}f(w_k)|^2}{|\Delta_{w_{k-1},\ldots,w_1}f(\sigma)-
\Delta_{w_{k-1},\ldots,w_1}f(w_k)|^2}-\frac{1-|w_k|^2}{|\sigma-w_k|^2}\;,
\end{aligned}
\]
and the non-tangential limit $\Delta_{w_{k},\ldots,w_1}f(\sigma)$ is inductively given by
\[
\Delta_{w_{k},\ldots,w_1}f(\sigma)=\overline{\Delta_{w_{k-1},\ldots,w_1}f(\sigma)}\sigma\frac{\Delta_{w_{k-1},\ldots,w_1}f(\sigma)-\Delta_{w_{k-1},\ldots,w_1}f(w_k)}{\overline{\Delta_{w_{k-1},\ldots,w_1}f(\sigma)}-\overline{\Delta_{w_{k-1},\ldots,w_1}f(w_k)}}\frac{\overline{\sigma}-\overline{w_k}}{\sigma-w_k}\;.
\]
In particular we have a multipoint Julia lemma:

\begin{Theorem}
\label{th:1.mpJulia}
Given $k\ge 1$, take $f\in\Hol(\D,\D)$ not a Blaschke product of degree at most~$k$. Let $\sigma\in\de\D$ be such that $\beta_f(\sigma)<+\infty$.
Then
\begin{equation}
\frac{|\Delta_{w_k,\ldots,w_1}f(\sigma)-\Delta_{w_k,\ldots,w_1}f(z)|^2}{1-|\Delta_{w_k,\ldots,w_1}f(z)|^2}\le\beta_{\Delta_{w_k,\ldots,w_1}f}(\sigma)\frac{|\sigma-z|^2}{1-|z|^2}
\label{eq:1.mpJulia}
\end{equation}
for all $z$, $w_1,\ldots,w_k\in\D$. %In particular,
%\[
%\frac{|f^*(\sigma,z)-f^h(z)|^2}{1-|f^h(z)|^2}\le\beta^*_f(\sigma;z)\frac{|\sigma-z|^2}{1-|z|^2}=\beta_f\,\frac{|\sigma-z|^2}{1-|z|^2}\bigg/\frac{|f(\sigma)-f(z)|^2}{1-|f(z)|^2}-1\;.
%\]
Moreover, equality occurs for some $z_0$,~$w_1,\ldots,w_k\in\D$ if and only if it occurs everywhere if and only if $f$ is a Blaschke product of degree~$k+1$.
\end{Theorem}

\begin{proof}
The inequality \eqref{eq:1.mpJulia} follows from Theorem~\ref{th:I.2.cinque} applied to $\Delta_{w_k,\ldots,w_1}f$. If we have equality in \eqref{eq:1.mpJulia}
for some $z_0$,~$w_1,\ldots,w_k\in\D$ again Theorem~\ref{th:I.2.cinque} implies that $\Delta_{w_k,\ldots,w_1}f\in\Aut(\D)$, and then Proposition~\ref{th:1.Blhdq} implies that
$f\in\mathcal{B}_{k+1}$. Conversely, $f\in\mathcal{B}_{k+1}$ implies that $\Delta_{w_k,\ldots,w_1}f\in\Aut(\D)$ for all $w_1,\ldots,w_k\in\D$, and thus we have equality in \eqref{eq:1.mpJulia}
for all $z$,~$w_1,\ldots,w_k\in\D$.
\end{proof}

The idea is that we can use this multipoint approach to improve known estimates by involving higher order derivatives. We shall show two examples of this: a strengthened version of Corollary~\ref{th:1.basso} and a generalization of Proposition~\ref{th:1.CPtrue}.

We begin with a reformulation of Theorem~\ref{th:1.mpJulia} which gives a far-reaching generalization of Corollary~\ref{th:1.basso}. %Notice that the case $k=0$ in the next statement is exactly the classical Julia lemma; so in a sense this is also a multipoint version of the classical Julia lemma.

\begin{Theorem}
\label{th:2.bassogen}
Given $k\ge 0$, let $f\in\Hol(\D,\D)$ be not a Blaschke product of degree at most~$k$. Take $\sigma\in\de\D$ with $\beta_f(\sigma)<+\infty$. Then
\begin{equation}
\beta_f(\sigma)\ge\sum_{j=0}^k\frac{1-|w_{j+1}|^2}{|\sigma-w_{j+1}|^2}\prod_{h=0}^j\frac{|\Delta_{w_h,\ldots,w_1}f(\sigma)-\Delta_{w_h,\ldots,w_1}f(w_{h+1})|^2}{1-|\Delta_{w_h,\ldots,w_1}f(w_{h+1})|^2}
\label{eq:2.bassogen}
\end{equation}
for every $w_1,\ldots,w_{k+1}\in\D$, where $\Delta_{w_h,\ldots,w_1}f=f$ when $h=0$. Furthermore we have equality in \eqref{eq:2.bassogen} for some $w_1,\ldots,w_{k+1}\in\D$ if and only if we have equality for all $w_1,\ldots,w_{k+1}\in\D$ if and only if $f$ is a Blaschke product of degree~$k+1$.
\end{Theorem}

\begin{proof}
One way to prove the assertion is to obtain by induction a formula for $\beta_{\Delta_{w_k,\ldots,w_1}f}(\sigma)$ applying repeatedly Proposition~\ref{th:1.beta}, and then to show that, with this formula, \eqref{eq:2.bassogen} is equivalent to \eqref{eq:1.mpJulia}. For the sake of variety we shall describe a different proof, relying on the
classical Julia lemma.

We proceed by induction on~$k$. The case $k=0$ is
\begin{equation}
\beta_f(\sigma)\ge\frac{1-|w_1|^2}{|\sigma-w_1|^2}\frac{|f(\sigma)-f(w_1)|^2}{1-|f(w_1)|^2}
\label{eq:2.basso0}
\end{equation}
which is exactly the classical Julia inequality \eqref{eq:I.2.nove}. In particular, we have equality for some $w_1\in\D$ (and hence for all $w_1\in\D$) if and only if $f\in\Aut(\D)$.

Assume that \eqref{eq:2.bassogen} holds for $k-1$, and take $w_1\in\D$. Since $f$ is not a Blaschke product of degree at most~$k$, by Proposition~\ref{th:1.Blhdq} $\Delta_{w_1}f$ is not a Blaschke product of degree at most $k-1$. So by the inductive hypothesis we have
\[
\begin{aligned}
\beta_{\Delta_{w_1}f}(\sigma)&\ge\sum_{j=1}^{k}\frac{1-|w_{j+1}|^2}{|\sigma-w_{j+1}|^2}\prod_{h=1}^j\frac{|\Delta_{w_h,\ldots,w_2}(\Delta_{w_1}f)(\sigma)-\Delta_{w_h,\ldots,w_2}(\Delta_{w_1}f)(w_{h+1})|^2}{1-|\Delta_{w_h,\ldots,w_2}(\Delta_{w_1}f)(w_{h+1})|^2}\\
&=\sum_{j=1}^{k}\frac{1-|w_{j+1}|^2}{|\sigma-w_{j+1}|^2}\prod_{h=1}^j\frac{|\Delta_{w_h,\ldots,w_1}f(\sigma)-\Delta_{w_h,\ldots,w_1}f(w_{h+1})|^2}{1-|\Delta_{w_h,\ldots,w_1}f(w_{h+1})|^2}
\end{aligned}
\]
for all $w_2,\ldots, w_{k+1}\in\D$, with equality for some (and hence all) $w_2,\ldots, w_{k+1}\in\D$ if and only if $\Delta_{w_1}f$ is a Blaschke product of degree~$k$, that is,
by Proposition~\ref{th:1.Blhdq}, if and only if $f$ is a Blaschke product of degree~$k+1$. 

Now Proposition~\ref{th:1.beta} yields
\[
\beta_f(\sigma)=\frac{|f(\sigma)-f(w_1)|^2}{1-|f(w_1)|^2}\left[\frac{1-|w_1|^2}{|\sigma-w_1|^2}+\beta_{\Delta_{w_1}f}(\sigma)\right]\;;
\]
therefore
\[
\begin{aligned}
\beta_f(\sigma)&\ge\frac{|f(\sigma)-f(w_1)|^2}{1-|f(w_1)|^2}\\
&\quad\times\!\Biggl[\frac{1-|w_1|^2}{|\sigma-w_1|^2}%\\
%&\phantom{\frac{|f(\sigma)-f(w_1)|^2}{1-|f(w_1)|^2}\Biggl[}\qquad
+\sum_{j=1}^{k}\frac{1-|w_{j+1}|^2}{|\sigma-w_{j+1}|^2}\prod_{h=1}^j\frac{|\Delta_{w_h,\ldots,w_1}f(\sigma)-\Delta_{w_h,\ldots,w_1}f(w_{h+1})|^2}{1-|\Delta_{w_h,\ldots,w_1}f(w_{h+1})|^2}\Biggr]\\
%&=\frac{|f(\sigma)-f(w_1)|^2}{1-|f(w_1)|^2}\frac{1-|w_1|^2}{|\sigma-w_1|^2}\\
%&\quad+\frac{|f(\sigma)-f(w_1)|^2}{1-|f(w_1)|^2}\sum_{j=1}^{k}\frac{1-|w_{j+1}|^2}{|\sigma-w_{j+1}|^2}\prod_{h=1}^j\frac{|\Delta_{w_h,\ldots,w_1}f(\sigma)-\Delta_{w_h,\ldots,w_1}f(w_{h+1})|^2}{1-|\Delta_{w_h,\ldots,w_1}f(w_{h+1})|^2}\\
&=\sum_{j=0}^{k}\frac{1-|w_{j+1}|^2}{|\sigma-w_{j+1}|^2}\prod_{h=0}^j\frac{|\Delta_{w_h,\ldots,w_1}f(\sigma)-\Delta_{w_h,\ldots,w_1}f(w_{h+1})|^2}{1-|\Delta_{w_h,\ldots,w_1}f(w_{h+1})|^2}\;,
\end{aligned}
\]
with equality for some (and hence all) $w_1,\ldots, w_{k+1}\in\D$ if and only if $f$ is a Blaschke product of degree~$k+1$, and we are done.
\end{proof}

\begin{Remark}
\label{rem:2.bassogen}
It is easy to see that
%Since
%\[
%\begin{aligned}
%\sum_{j=0}^k&\frac{1-|w_{j+1}|^2}{|\sigma-w_{j+1}|^2}\prod_{h=0}^j\frac{|\Delta_{w_h,\ldots,w_1}f(\sigma)-\Delta_{w_h,\ldots,w_1}f(w_{h+1})|^2}{1-|\Delta_{w_h,\ldots,w_1}f(w_{h+1})|^2}\\
%&=\frac{1-|w_1|^2}{|\sigma-w_1|^2}\frac{|f(\sigma)-f(w_1)|^2}{1-|f(w_1)|^2}+\frac{1-|w_2|^2}{|\sigma-w_2|^2}\frac{|\Delta_{w_1}f(\sigma)-\Delta_{w_1}f(w_2)|}{1-|\Delta_{w_1}f(w_2)|^2}\frac{|f(\sigma)-f(w_1)|^2}{1-|f(w_1)|^2}\\
%&\quad+\cdots+\frac{1-|w_{k+1}|^2}{|\sigma-w_{k+1}|^2}\frac{|\Delta_{w_{k},\ldots,w_1}f(\sigma)-\Delta_{w_{k},\ldots,w_1}f(w_{k+1})|^2}{1-|\Delta_{w_{k},\ldots,w_1}f(w_{k+1})|^2}\cdots\frac{|f(\sigma)-f(w_1)|^2}{1-|f(w_1)|^2}\\
%&\ge\frac{1-|w_1|^2}{|\sigma-w_1|^2}\frac{|f(\sigma)-f(w_1)|^2}{1-|f(w_1)|^2}+\frac{1-|w_2|^2}{|\sigma-w_2|^2}\frac{|\Delta_{w_1}f(\sigma)-\Delta_{w_1}f(w_2)|}{1-|\Delta_{w_1}f(w_2)|^2}\frac{|f(\sigma)-f(w_1)|^2}{1-|f(w_1)|^2}\\
%&\quad+\cdots+\frac{1-|w_{k}|^2}{|\sigma-w_{k}|^2}\frac{|\Delta_{w_{k-1},\ldots,w_1}f(\sigma)-\Delta_{w_{k-1},\ldots,w_1}f(w_{k})|^2}{1-|\Delta_{w_{k-1},\ldots,w_1}f(w_{k})|^2}\cdots\frac{|f(\sigma)-f(w_1)|^2}{1-|f(w_1)|^2}\\
%&\ge\cdots\\
%&\ge\frac{1-|w_1|^2}{|\sigma-w_1|^2}\frac{|f(\sigma)-f(w_1)|^2}{1-|f(w_1)|^2}+\frac{1-|w_2|^2}{|\sigma-w_2|^2}\frac{|\Delta_{w_1}f(\sigma)-\Delta_{w_1}f(w_2)|}{1-|\Delta_{w_1}f(w_2)|^2}\frac{|f(\sigma)-f(w_1)|^2}{1-|f(w_1)|^2}\\
%&\ge \frac{1-|w_1|^2}{|\sigma-w_1|^2}\frac{|f(\sigma)-f(w_1)|^2}{1-|f(w_1)|^2}\;,
%\end{aligned}
%\]
the estimate \eqref{eq:2.bassogen} becomes better and better as $k$ increases.
\end{Remark}

Theorem~\ref{th:2.bassogen} has a number of corollaries that it is worthwhile to state.

\begin{Corollary}
\label{th:2.bassogenuno}
Given $k\ge 0$, let $f\in\Hol(\D,\D)$ be not a Blaschke product of degree at most~$k$. Take $\sigma\in\de\D$ with $\beta_f(\sigma)<+\infty$. Then
\begin{equation}
\beta_f(\sigma)\ge\sum_{j=0}^k\frac{1-|w_{j+1}|^2}{|\sigma-w_{j+1}|^2}\prod_{h=0}^j\frac{1-|\Delta_{w_h,\ldots,w_1}f(w_{h+1})|}{1+|\Delta_{w_h,\ldots,w_1}f(w_{h+1})|}
\label{eq:2.bassogenuno}
\end{equation}
for every $w_1,\ldots,w_{k+1}\in\D$, where $\Delta_{w_h,\ldots,w_1}f=f$ when $h=0$. Furthermore we have equality in \eqref{eq:2.bassogenuno} if and only if $f$ is a Blaschke product of degree~$k+1$ and $w_1,\ldots,w_{k+1}\in\D$ are such that
\[
\Delta_{w_h,\ldots,w_1}f(w_{h+1})=|\Delta_{w_h,\ldots,w_1}f(w_{h+1})|\Delta_{w_h,\ldots,w_1}f(\sigma)
\]
for all $h=0,\ldots,k$.
\end{Corollary}

\begin{proof}
It follows from \eqref{eq:2.bassogen} using the standard estimate $|\tau-z|\ge 1-|z|$ valid for all $\tau\in\de\D$ and $z\in\D$, with equality if and only if $z=|z|\tau$.
\end{proof}

\begin{Remark}
\label{rem:2.Mercer2}
If take $k=1$ and $w_1=w_2=z$ and we assume $\sigma=f(\sigma)=1$ then \eqref{eq:2.bassogenuno} becomes exactly \cite{Mercer2018}*{Theorem~2.1}.
\end{Remark}

\begin{Corollary}
\label{th:2.bassot}
Given $k\ge 0$ let $f\in\Hol(\D,\D)$ be not a Blaschke product of degree at most~$k$. Take $\sigma\in\de\D$ with $\beta_f(\sigma)<+\infty$. Then
\begin{equation}
\begin{aligned}
\beta_f(\sigma)&\ge\sum_{j=0}^k\prod_{h=0}^j\frac{|\Delta_{\mathbf{O}_h}f(\sigma)-\Delta_{\mathbf{O}_h}f(0)|^2}{1-|\Delta_{\mathbf{O}_h}f(0)|^2}
&\ge\sum_{j=0}^k\prod_{h=0}^j\frac{1-|\Delta_{\mathbf{O}_h}f(0)|}{1+|\Delta_{\mathbf{O}_h}f(0)|}
\end{aligned}
\label{eq:2.bassot}
\end{equation}
where $\mathbf{O}_h=(0,\ldots,0)\in\D^h$ is the origin of~$\C^h$, and $\Delta_{\mathbf{O}_h}f=f$ when $h=0$. Furthermore we have equality on the left of \eqref{eq:2.bassot} if and only if $f$ is a Blaschke product of degree~$k+1$, and on the right if and only if 
\[
\Delta_{\mathbf{O}_h}f(0)=|\Delta_{\mathbf{O}_h}f(0)|\Delta_{\mathbf{O}_h}f(\sigma)
\]
for all $h=0,\ldots,k$.
\end{Corollary}

\begin{proof}
It follows from Theorem~\ref{th:2.bassogen} taking $w_h=0$ for $h=1,\ldots,k+1$.
\end{proof}

\begin{Remark}
\label{rem:2.bassot}
\eqref{eq:2.bassot} for $k=0$ is exactly \eqref{eq:I.2.basso}, while for $k=1$ it yields \eqref{eq:1.Oss}, because $\Delta_0f(0)=f^h(0)$. \end{Remark}

\begin{Corollary}
\label{th:2.bassoinf}
Let $f\in\Hol(\D,\D)$ be not a Blaschke product, and $\sigma\in\de\D$ with $\beta_f(\sigma)<+\infty$. Then
\begin{equation}
\begin{aligned}
\label{eq:2.bassoinf}
\beta_f(\sigma)&\ge\sum_{j=0}^\infty\frac{1-|w_{j+1}|^2}{|\sigma-w_{j+1}|^2}\prod_{h=0}^j\frac{|\Delta_{w_h,\ldots,w_1}f(\sigma)-\Delta_{w_h,\ldots,w_1}f(w_{h+1})|^2}{1-|\Delta_{w_h,\ldots,w_1}f(w_{h+1})|^2}\\
&\ge\sum_{j=0}^\infty\frac{1-|w_{j+1}|^2}{|\sigma-w_{j+1}|^2}\prod_{h=0}^j\frac{1-|\Delta_{w_h,\ldots,w_1}f(w_{h+1})|}{1+|\Delta_{w_h,\ldots,w_1}f(w_{h+1})|}
\end{aligned}
\end{equation}
for any sequence $\{w_h\}\subset\D$, where $\Delta_{w_h,\ldots,w_1}f=f$ when $h=0$ as usual. In particular,
\[
\begin{aligned}
\beta_f(\sigma)&\ge\sum_{j=0}^\infty\prod_{h=0}^j\frac{|\Delta_{\mathbf{O}_h}f(\sigma)-\Delta_{\mathbf{O}_h}f(0)|^2}{1-|\Delta_{\mathbf{O}_h}f(0)|^2}
&\ge\sum_{j=0}^\infty\prod_{h=0}^j\frac{1-|\Delta_{\mathbf{O}_h}f(0)|}{1+|\Delta_{\mathbf{O}_h}f(0)|}\;.
\end{aligned}
\]
\end{Corollary}

\begin{proof}
It follows from Theorem~\ref{th:2.bassogen}, Corollary~\ref{th:2.bassot} and Remark~\ref{rem:2.bassogen}.
\end{proof}

A natural question, that we leave open, is whether the first inequality in \eqref{eq:2.bassoinf} actually is an equality, at least when $f$ is an infinite Blaschke product. 

To give an idea of the actual content of \eqref{eq:2.bassogen}, let us reformulate it for small values of~$k$ and particular values of~$w_1,\ldots,w_{k+1}$. 

For $k=0$ we get \eqref{eq:2.basso0}, that we already noticed to be equivalent to the classical Julia lemma.

For $k=1$ we get
\[
\beta_f(\sigma)\ge\frac{|f(\sigma)-f(w_1)|^2}{1-|f(w_1)|^2}\left[\frac{1-|w_1|^2}{|\sigma-w_1|^2}+\frac{1-|w_2|^2}{|\sigma-w_2|^2}\frac{|\Delta_{w_1}f(\sigma)-\Delta_{w_1}f(w_2)|^2}{1-|\Delta_{w_1}f(w_2)|^2}\right]\;.
\]
Since $\Delta_0f(0)=f^h(0)$ and
\begin{equation}
\Delta_0f(\sigma)=\overline{f(\sigma)}\overline{\sigma}\frac{f(\sigma)-f(0)}{\overline{f(\sigma)}-\overline{f(0)}}\;,
\label{eq:2.Ds0}
\end{equation}
putting $w_1=w_2=0$ we obtain
\[
\beta_f(\sigma)\ge\frac{|f(\sigma)-f(0)|^2}{1-|f(0)|^2}\left[1+\frac{\left|\overline{f(\sigma)}\overline{\sigma}\frac{f(\sigma)-f(0)}{\overline{f(\sigma)}-\overline{f(0)}}-f^h(0)\right|^2}{1-|f^h(0)|^2}\right]\;,
\]
which is a slightly more precise version of \eqref{eq:1.Oss}. If moreover $f(0)=0$ we find again~\eqref{eq:1.Jzz}.
%\[
%\beta_f(\sigma)\ge 1+\frac{\left|\frac{f(\sigma)}{\sigma}-f'(0)\right|^2}{1-|f'(0)|^2}\;.
%\]

The case $k=2$ with $w_1=w_2=w_3=0$ yields
\begin{equation}
\begin{aligned}
\beta_f(\sigma)&\ge\frac{|f(\sigma)-f(0)|^2}{1-|f(0)|^2}\\
&\quad\times\left[1+\frac{\left|\overline{f(\sigma)}\overline{\sigma}\frac{f(\sigma)-f(0)}{\overline{f(\sigma)}-\overline{f(0)}}-f^h(0)\right|^2}{1-|f^h(0)|^2}
\left(1+\frac{|\Delta_{0,0}f(\sigma)-(\Delta_0f)^h(0)|^2}{1-|(\Delta_0f)^h(0)|^2}\right)\right]\;,
\end{aligned}
\label{eq:2.basso2}
\end{equation}
where we have used the equality $\Delta_{0,0}f(0)=(\Delta_0f)^h(0)$.
To compute $(\Delta_0f)^h(0)$ first all we notice that
\[
\begin{aligned}
(\Delta_{w_0}f)'(z)&=\frac{f'(z)(z-w_0)-\bigl(f(z)-f(w_0)\bigr)}{(z-w_0)^2}\cdot\frac{1-\overline{w_0}z}{1-\overline{f(w_0)}f(z)}\\
&\quad+\frac{f(z)-f(w_0)}{z-w_0}\cdot\frac{-\overline{w_0}\bigl(1-\overline{f(w_0)}f(z)\bigr)+(1-\overline{w_0}z)\overline{f(w_0)}f'(z)}{\bigl(1-\overline{f(w_0)}f(z)\bigr)^2}\;,
\end{aligned}
\]
and so
\[
\begin{aligned}
(\Delta_{w_0}f)^h(0)=\frac{1}{1-|\Delta_{w_0}f(0)|^2}\biggl[&\frac{f(w_0)-f(0)-f'(0)w_0}{w_0^2}\cdot\frac{1}{1-\overline{f(w_0)}f(0)}\\
&+\frac{f(w_0)-f(0)}{w_0}\frac{\overline{f(w_0)}f'(0)-\overline{w_0}\bigl(1-\overline{f(w_0)}f(0)\bigr)}{\bigl(1-\overline{f(w_0)}f(0)\bigr)^2}\biggr].
\end{aligned}
\]
In particular putting $w_0=0$ we get
\begin{equation}
\begin{aligned}
(\Delta_{0}f)^h(0)&=\frac{1}{1-|f^h(0)|^2}\biggl[\frac{f''(0)}{2\bigl(1-|f(0)|^2\bigr)}+\frac{\overline{f(0)}f'(0)^2}{\bigl(1-|f(0)|^2\bigr)^2}\biggr]\\
&=\frac{1}{1-|f^h(0)|^2}\biggl[\frac{f''(0)}{2\bigl(1-|f(0)|^2\bigr)}+\overline{f(0)}f^h(0)^2\biggr]\;.
\end{aligned}
\label{eq:2.dzh}
\end{equation}
Applying \eqref{eq:2.Ds0} to $\Delta_0f$ we also get
\[
\Delta_{0,0}f(\sigma)=f(\sigma)\frac{\overline{f(\sigma)}-\overline{f(0)}}{f(\sigma)-f(0)}\frac{\Delta_0f(\sigma)-f^h(0)}{\overline{\Delta_0f(\sigma)}-\overline{f^h(0)}}\;,
\]
and we have all the terms appearing in \eqref{eq:2.basso2}. In particular,
\[
\beta_f(\sigma)\ge\frac{1-|f(0)|}{1+|f(0)|}\left[1+\frac{1-|f^h(0)|}{1+|f^h(0)|}\frac{2}{1+|(\Delta_0f)^h(0)|}\right]\;,
\]
and thus if $f(0)=0$ we obtain
\[
\beta_f(\sigma)\ge 1+\frac{2(1-|f'(0)|)}{1+|f'(0)|+\frac{|f''(0)|}{2(1-|f'(0)|)}}
\]
that improves \eqref{eq:1.Jzz}. 
If moreover $f'(0)=0$ we also get 
\[
\beta_f(\sigma)\ge 1+\frac{2}{1+\frac{1}{2}|f''(0)|}\ge 2
\]
because $\frac{1}{2}|f''(0)|=|\Delta_{0,0}f(0)|\le 1$. 

As a final example, we record that a similar argument with $k=3$ and $w_1=w_2=w_3=w_4=0$ yields
\[
\beta_f(\sigma)\ge\frac{1-|f(0)|}{1+|f(0)|}\biggl[\frac{1-|f^h(0)|}{1+|f^h(0)|}
\left(\frac{1-|(\Delta_{0}f)^h(0)|}{1+|(\Delta_{0}f)^h(0)|}\frac{2}{1+|(\Delta_{0,0}f)^h(0)|}+1\right)+1\biggr]\;,
\]
where
\[
(\Delta_{0,0}f)^h(0)=\frac{1}{1-|(\Delta_{0}f)^h(0)|^2}\biggl[\frac{(\Delta_{0}f)''(0)}{2\bigl(1-|f^h(0)|^2\bigr)}+\overline{f^h(0)}(\Delta_{0}f)^h(0)^2\biggr]\;,
\]
with
\[
(\Delta_{0}f)''(0)=\frac{1}{1-|f(0)|^2}\biggl[\frac{1}{3}f'''(0)+2\overline{f(0)}f^h(0)f''(0)+\overline{f(0)}^2f^h(0)^2f'(0)\biggr]\;.
\]

We now proceed toward the promised generalization of Proposition~\ref{th:1.CPtrue}. 
Let us start with the following reformulation of the case $k=2$ of Theorem~\ref{th:1.mpJulia} valid when~$f(0)=0$:

\begin{Proposition}
\label{th:2.CP}
Let $f\in\Hol(\D,\D)$, not a Blaschke product of degree at most~2, and $\sigma\in\de\D$ be such that $\beta_f(\sigma)<+\infty$.
Assume that $f(0)=0$. Then 
\begin{equation}
\frac{\left|\frac{\overline{\sigma}\bigl(f(\sigma)\overline{\sigma}-f'(0)\bigr)}{1-\overline{f'(0)}f(\sigma)\overline{\sigma}}-\frac{\frac{1}{z}\left(\frac{f(z)}{z}-f'(0)\right)}{1-\overline{f'(0)}\frac{f(z)}{z}}\right|^2}{1-\left|\frac{\frac{1}{z}\left(\frac{f(z)}{z}-f'(0)\right)}{1-\overline{f'(0)}\frac{f(z)}{z}}\right|^2}\le\left[\frac{1-|f'(0)|^2}{\left|\frac{f(\sigma)}{\sigma}-f'(0)\right|^2}\bigl(\beta_f(\sigma)-1\bigr)-1\right]\!\frac{|\sigma-z|^2}{1-|z|^2}
\label{eq:2.Jz}
\end{equation}
for $z\ne 0$, and
\begin{equation}
\frac{\left|\frac{\overline{\sigma}\bigl(f(\sigma)\overline{\sigma}-f'(0)\bigr)}{1-\overline{f'(0)}f(\sigma)\overline{\sigma}}-\frac{f''(0)}{2\bigl(1-|f'(0)|^2\bigr)}\right|^2}{1-\left|\frac{f''(0)}{2\bigl(1-|f'(0)|^2\bigr)}\right|^2}\le\frac{1-|f'(0)|^2}{\left|\frac{f(\sigma)}{\sigma}-f'(0)\right|^2}\bigl(\beta_f(\sigma)-1\bigr)-1\;.
\label{eq:2.Jzz}
\end{equation}
%In particular, if furthermore $f(\sigma)=\sigma$ we get
%\begin{equation}
%\frac{|1-f'(0)|^2}{1-|f'(0)|^2}\le\beta_f(\sigma)-1\;.
%\label{eq:1.Jzzz}
%\end{equation}
Moreover, equality occurs in \eqref{eq:2.Jz} at some $z_0\in\D\setminus\{0\}$ or in \eqref{eq:2.Jzz} if and only if it always occurs if and only $f$ is a Blaschke product of degree~3.
\end{Proposition}

\begin{proof}
We would like to apply Theorem~\ref{th:1.mpJulia} with $k=2$ and $w_1=w_2=0$. 

First of all
\[
%\begin{aligned}
\beta_{\Delta_{0,0}f}(\sigma)=\beta_{\Delta_0f}(\sigma)\frac{1-|\Delta_0f(0)|^2}{|\Delta_0f(\sigma)-\Delta_0f(0)|^2}-1%\\
=\frac{1-|f'(0)|^2}{\left|\frac{f(\sigma)}{\sigma}-f'(0)\right|^2}\bigl(\beta_f(\sigma)-1\bigr)-1\;,
%\end{aligned}
\]
where we used $f(0)=0$ and $\Delta_0f(0)=f^h(0)=f'(0)$. Next recalling \eqref{eq:1.recall} we get
%\[
%\Delta_0f(z)=\begin{cases}
%\frac{f(z)}{z}&\mathrm{if}\ z\ne0\;,\\
%f'(0)&\mathrm{if}\ z=0\;,
%\end{cases}
%\]
%and
\[
\Delta_{0,0}f(z)=\begin{cases}
\frac{\frac{f(z)}{z}-f'(0)}{z}\frac{1}{1-\overline{f'(0)}\frac{f(z)}{z}}&\mathrm{if}\ z\ne 0\;,\\
\frac{f''(0)}{2(1-|f'(0)|^2)}&\mathrm{if}\ z=0
\end{cases}
\]
where we used \eqref{eq:2.dzh}. The assertion then follows from Theorem~\ref{th:1.mpJulia}.
\end{proof}

Notice that \eqref{eq:1.Jzzz} is equivalent to saying that the right-hand side of \eqref{eq:2.Jzz} is non-negative; so \eqref{eq:2.Jzz} is an improvement of  \eqref{eq:1.Jzzz}.

If $f(0)=f'(0)=0$ \eqref{eq:2.Jz} and \eqref{eq:2.Jzz} simplify becoming
\[
\frac{\left|\frac{f(\sigma)}{\sigma^2}-\frac{f(z)}{z^2}\right|^2}{1-\left|\frac{f(z)}{z^2}\right|^2}\le\bigl(\beta_f(\sigma)-2\bigr)\frac{|\sigma-z|^2}{1-|z|^2}
\]
for $z\ne 0$ and
\[
\frac{\left|\frac{f(\sigma)}{\sigma^2}-\frac{1}{2}f''(0)\right|^2}{1-\left|\frac{1}{2}f''(0)\right|^2}\le\beta_f(\sigma)-2\;.
\]
These formulas suggest the following

\begin{Proposition}
\label{th:2.CPn}
Let $f\in\Hol(\D,\D)$ and $\sigma\in\de\D$ be such that $\beta_f(\sigma)<+\infty$. Given $k\ge 1$, assume that $f$ is not a Blaschke product of degree at most~$k\ge1$,
and that $f(0)=\cdots=f^{(k-1)}(0)=0$. Then 
\begin{equation}
\frac{\left|\frac{f(\sigma)}{\sigma^k}-\frac{f(z)}{z^k}\right|^2}{1-\left|\frac{f(z)}{z^k}\right|^2}\le\bigl(\beta_f(\sigma)-k\bigr)\frac{|\sigma-z|^2}{1-|z|^2}
\label{eq:2.Jzk}
\end{equation}
for $z\ne 0$ and
\begin{equation}
\frac{\left|\frac{f(\sigma)}{\sigma^k}-\frac{1}{k!}f^{(k)}(0)\right|^2}{1-\left|\frac{1}{k!}f^{(k)}(0)\right|^2}\le\beta_f(\sigma)-k\;.
\label{eq:2.Jzzk}
\end{equation}
Moreover, equality occurs in \eqref{eq:2.Jzk} at some $z_0\in\D\setminus\{0\}$ or in \eqref{eq:2.Jzzk} if and only if it always occurs if and only $f$ is a Blaschke product of degree~$k+1$.
\end{Proposition}

\begin{proof}
%There are (at least) two ways of proving this statement. Possibly the easiest one is to apply Corollary~\ref{th:1.CP} to $g(z)=f(z)/z^{k-1}$.  It works because $g(\sigma)=\overline{\sigma}^{k-1}f(\sigma)$ and $g'(\sigma)=f(\sigma)\overline{\sigma}^k[\beta_f-(k-1)]$.
%
By induction it is easy to prove that
\begin{equation}
\Delta_{\mathbf{O}_k}f(z)=
\begin{cases}
\frac{f(z)}{z^k}&\mathrm{if}\ z\ne0\;,\\
\frac{1}{k!}f^{(k)}(0)&\mathrm{if}\ z=0\;,
\end{cases}
\label{eq:2.Ok}
\end{equation}
and that $\beta_{\Delta_{\mathbf{O}_k}f}(\sigma)=\beta_f(\sigma)-k$. The assertion then follows from Theorem~\ref{th:1.mpJulia}.
\end{proof}

In particular, if $f(\sigma)=\sigma^k$ the left-hand sides of \eqref{eq:2.Jzk} and \eqref{eq:2.Jzzk} become independent of~$\sigma$. This suggests that we might obtain 
a generalization of Proposition~\ref{th:1.CPtrue} with multiple fixed points. It turns out that this is easy when the multiple fixed point is the origin (see Corollary~\ref{th:2.CPgen} below), but the statement 
and the proof of the general result when the multiple fixed point is not the origin are considerably harder:

\begin{Theorem}
\label{th:2.CPgenz}
Let $f\in\Hol(\D,\D)$. Given $k\ge 1$, assume that $f$ is not a Blaschke product of degree at most~$k$ and that there exists $z_0\in\D$ 
such that $f(z_0)=z_0$ and $f'(z_0)=\ldots=f^{(k-1)}(z_0)=0$. 
Take $\sigma_1,\ldots,\sigma_n\in\de\D$ distinct points such that $\beta_f(\sigma_j)<+\infty$ and
\begin{equation}
%f(\sigma_j)=\sigma_j^k\left(\frac{1-\overline{\sigma_j}z_0}{1-\sigma_j\overline{z_0}}\right)^{\!k}\,\frac{1+\left(\frac{1-\overline{z_0}\sigma_j}{\sigma_j-z_0}\right)^k z_0}{1+\overline{\left(\frac{1-z_0\overline{\sigma_j}}{\overline{\sigma_j}-\overline{z_0}}\right)^k \overline{z_0}}}
f(\sigma_j)=\frac{\left(\frac{\sigma_j-z_0}{1-\overline{z_0}\sigma_j}\right)^k+z_0}{1+\overline{z_0}\left(\frac{\sigma_j-z_0}{1-\overline{z_0}\sigma_j}\right)^k}\;,
\label{eq:2.Dksigma}
\end{equation}
for $j=1,\ldots,n$. %, where $(\mathbf{z_0})_k=(z_0,\ldots,z_0)\in\D^k$. 
Then
\begin{equation}
\sum_{j=1}^{n}\frac{1}{\left(1+2\Re\frac{\left(f(\sigma_j)-\sigma_j\right)\overline{z_0}}{|f(\sigma_j)-z_0|^2}\right)\beta_f(\sigma_j)-k}\le\frac{1-\left|\frac{f^{(k)}(z_0)}{k!}(1-|z_0|^2)^{k-1}\right|^2}{\left|1-\frac{f^{(k)}(z_0)}{k!}(1-|z_0|^2)^{k-1}\right|^2}\;,
\label{eq:2.CPzz}
\end{equation}
with equality if and only if $f$ is a Blaschke product of degree~$n+k$.
\end{Theorem}

\begin{proof}
The fact that $z_0$ is multiple fixed point of~$f$ of order~$k$ is equivalent to saying that we can write
\[
f(z)=z_0+\frac{1}{k!}f^{(k)}(z_0)(z-z_0)^k+O\bigl((z-z_0)^{k+1}\bigr)\;.
\]
We claim that then
\begin{equation}
\Delta_{(\mathbf{z_0})_h}f(z)=\frac{1}{k!}f^{(k)}(z_0)(1-|z_0|^2)^{h-1}(z-z_0)^{k-h}+O\bigl((z-z_0)^{k-h+1}\bigr)
\label{eq:2.CPzuno}
\end{equation}
for all $h=1,\ldots,k$, where $(\mathbf{z_0})_h=(z_0,\ldots,z_0)\in\D^h$. We proceed by induction. For $h=1$ we have
\[
\begin{aligned}
\Delta_{z_0}f(z)&=\frac{f(z)-z_0}{z-z_0}\frac{1-\overline{z_0}z}{1-\overline{z_0}f(z)}=\frac{1}{k!}f^{(k)}(z_0)(z-z_0)^{k-1}\bigl[1+O(z-z_0)\bigr]\\
&=\frac{1}{k!}f^{(k)}(z_0)(z-z_0)^{k-1}+O\bigl((z-z_0)^k\bigr)\;,
\end{aligned}
\]
as claimed. Assume that \eqref{eq:2.CPzuno} holds for $1\le h-1\le k-1$. Then $\Delta_{(\mathbf{z_0})_{h-1}}f(z_0)=0$ yields
\[
\begin{aligned}
\Delta_{(\mathbf{z_0})_h}f(z)&=\frac{\Delta_{(\mathbf{z_0})_{h-1}}f(z)}{z-z_0}(1-\overline{z_0}z)\\
&=\frac{1}{k!}f^{(k)}(z_0)(1-|z_0|^2)^{h-2}(z-z_0)^{k-h}\bigl[1-|z_0|^2
+O(z-z_0)\bigr]\\
&=\frac{1}{k!}f^{(k)}(z_0)(1-|z_0|^2)^{h-1}(z-z_0)^{k-h}+O\bigl((z-z_0)^{k-h+1}\bigr)\;,
\end{aligned}
\]
as wanted. 

Now we claim that
\begin{equation}
\beta_{\Delta_{(\mathbf{z_0})_h}f}(\sigma_j)=\frac{1-|z_0|^2}{|\sigma_j-z_0|^2}\left[\left(1+2\Re\frac{\bigl(f(\sigma_j)-\sigma_j\bigr)\overline{z_0}}{|f(\sigma_j)-z_0|^2}
\right)\beta_f(\sigma_j)-h\right]
\label{eq:2.CPzdue}
\end{equation}
for all $h=1,\ldots,k$ and $j=1,\ldots,n$. We again proceed by induction on $h$. For $h=1$ Proposition~\ref{th:1.beta} yields
\[
\begin{aligned}
\beta_{\Delta_{z_0}f}(\sigma_j)&=\frac{1-|z_0|^2}{|f(\sigma_j)-z_0|^2}\beta_f(\sigma_j)-\frac{1-|z_0|^2}{|\sigma_j-z_0|^2}\\
&=\frac{1-|z_0|^2}{|\sigma_j-z_0|^2}\bigl(\beta_f(\sigma)-1\bigr)+\left(\frac{1-|z_0|^2}{|f(\sigma_j)-z_0|^2}-\frac{1-|z_0|^2}{|\sigma_j-z_0|^2}\right)\beta_f(\sigma)\\
&=\frac{1-|z_0|^2}{|\sigma_j-z_0|^2}\left(\beta_f(\sigma)-1 +\frac{|\sigma_j-z_0|^2-|f(\sigma_j)-z_0|^2}{|f(\sigma_j)-z_0|^2}\beta_f(\sigma_j)\right)\\
&=\frac{1-|z_0|^2}{|\sigma_j-z_0|^2}\left[\left(1+2\Re\frac{\bigl(f(\sigma_j)-\sigma_j\bigr)\overline{z_0}}{|f(\sigma_j)-z_0|^2}
\right)\beta_f(\sigma_j)-1\right]\;,
\end{aligned}
\]
as claimed. Assume that \eqref{eq:2.CPzdue} holds for $1\le h-1\le k-1$. Using again the fact that $\Delta_{(\mathbf{z_0})_{h-1}}f(z_0)=0$ we get
\[
\begin{aligned}
\beta_{\Delta_{(\mathbf{z_0})_h}f}(\sigma_j)&=\beta_{\Delta_{(\mathbf{z_0})_{h-1}}f}(\sigma_j)-\frac{1-|z_0|^2}{|\sigma_j-z_0|^2}\\
%&=\frac{1-|z_0|^2}{|\sigma_j-z_0|^2}\left[\left(1+2\Re\frac{\bigl(f(\sigma_j)-\sigma_j\bigr)\overline{z_0}}{|f(\sigma_j)-z_0|^2}
%\right)\beta_f(\sigma_j)-(h-1)\right]-\frac{1-|z_0|^2}{|\sigma_j-z_0|^2}\\
&=\frac{1-|z_0|^2}{|\sigma_j-z_0|^2}\left[\left(1+2\Re\frac{\bigl(f(\sigma_j)-\sigma_j\bigr)\overline{z_0}}{|f(\sigma_j)-z_0|^2}
\right)\beta_f(\sigma_j)-h\right]\;,
\end{aligned}
\]
and we are done.

We need one more preliminary computation. We claim that 
\begin{equation}
\Delta_{(\mathbf{z_0})_h}f(\sigma_j)=\overline{f(\sigma_j)}\sigma_j^h\left(\frac{\overline{\sigma_j}-\overline{z_0}}{\sigma_j-z_0}\right)^h\frac{f(\sigma_j)-z_0}{\overline{f(\sigma_j)}-\overline{z_0}}
\label{eq:2.Dhs}
\end{equation}
for $j=1,\ldots,n$ and $h=1,\ldots,k$. As always, we argue by induction on~$h$. For $h=1$ we have
\[
\Delta_{z_0}f(\sigma_j)=\frac{f(\sigma_j)-z_0}{1-\overline{z_0}f(\sigma_j)}\frac{1-\overline{z_0}\sigma_j}{\sigma_j-z_0}=
\overline{f(\sigma_j)}\sigma_j\frac{\overline{\sigma_j}-\overline{z_0}}{\sigma_j-z_0}\frac{f(\sigma_j)-z_0}{\overline{f(\sigma_j)}-\overline{z_0}}
\]
as claimed. Assume that \eqref{eq:2.Dhs} holds for $1\le h-1\le k-1$. Recalling that $\Delta_{(\mathbf{z_0})_{h-1}}f(z_0)=0$ we obtain
\[
\Delta_{(\mathbf{z_0})_h}f(\sigma_j)=\Delta_{(\mathbf{z_0})_{h-1}}f(\sigma_j)\frac{1-\overline{z_0}\sigma_j}{\sigma_j-z_0}=\overline{f(\sigma_j)}\sigma_j^h\left(\frac{\overline{\sigma_j}-\overline{z_0}}{\sigma_j-z_0}\right)^h\frac{f(\sigma_j)-z_0}{\overline{f(\sigma_j)}-\overline{z_0}}\;,
\]
and \eqref{eq:2.Dhs} is proved. In particular, we have $\Delta_{(\mathbf{z_0})_k}f(\sigma_j)=1$ if and only if
\[
\frac{1-\overline{f(\sigma_j)}z_0}{\overline{f(\sigma_j)}-\overline{z_0}}=\left(\frac{\sigma_j-z_0}{1-\overline{z_0}\sigma_j}\right)^k
\]
if and only if
\[
f(\sigma_j)=\frac{\left(\frac{\sigma_j-z_0}{1-\overline{z_0}\sigma_j}\right)^k+z_0}{1+\overline{z_0}\left(\frac{\sigma_j-z_0}{1-\overline{z_0}\sigma_j}\right)^k}\;.
\]
In other words, condition \eqref{eq:2.Dksigma} is just another way of writing $\Delta_{(\mathbf{z_0})_k}f(\sigma_j)=1$.

We can now apply Theorem~\ref{th:1.mpJulia}. Recalling the assumption $\Delta_{(\mathbf{z_0})_k}f(\sigma_j)=1$ we get
\begin{equation}
\frac{|1-\Delta_{(\mathbf{z_0})_k}f(z)|^2}{1-|\Delta_{(\mathbf{z_0})_k}f(z)|^2}\le \frac{1-|z_0|^2}{|\sigma_j-z_0|^2}\left[\left(1+2\Re\frac{\bigl(f(\sigma_j)-\sigma_j\bigr)\overline{z_0}}{|f(\sigma_j)-z_0|^2}\right)\beta_f(\sigma_j)-k\right]\frac{|\sigma_j-z|^2}{1-|z|^2}
\label{eq:2.CPztre}
\end{equation}
for all $z\in\D$ and $j=1,\ldots, n$, where 
\[
\Delta_{(\mathbf{z_0})_k}f(z)=\frac{1}{k!}f^{(k)}(z_0)(1-|z_0|^2)^{k-1}+O(z-z_0)
\]
by \eqref{eq:2.CPzuno}. Furthermore, equality in \eqref{eq:2.CPztre} holds in one point (and hence everywhere) if and only if $f$ is a Blaschke product of degree~$k+1$.

We now claim that
\begin{equation}
\begin{aligned}
\Re\biggl(&\frac{1+\Delta_{(\mathbf{z_0})_k}f(z)}{1-\Delta_{(\mathbf{z_0})_k}f(z)}\biggr)\\
&\ge\sum_{j=1}^n\frac{|\sigma_j-z_0|^2}{1-|z_0|^2}
\frac{1}{\left(1+2\Re\frac{\left(f(\sigma_j)-\sigma_j\right)\overline{z_0}}{|f(\sigma_j)-z_0|^2}\right)\beta_f(\sigma_j)-k}\Re\left(\frac{\sigma_j+z}{\sigma_j-z}\right)\;,
\end{aligned}
\label{eq:2.CPzqua}
\end{equation}
with equality in one point (and hence everywhere) if and only if $f$ is a Blaschke product of degree~$n+k$. 

We argue by induction on~$n$. For $n=1$ \eqref{eq:2.CPzqua} is exactly equivalent to~\eqref{eq:2.CPztre}.
Assume that \eqref{eq:2.CPzqua} holds for $n-1$. In particular we have
\begin{equation}
\Re\left[\frac{1+g(z)}{1-g(z)}-\sum_{j=1}^{n-1} a_j\frac{\sigma_j+z}{\sigma_j-z}\right]\ge 0\;,
\label{eq:2.CPzcin}
\end{equation}
with equality in one point (and hence everywhere) if and only if $f\in\mathcal{B}_{n+k-1}$, where $g=\Delta_{(\mathbf{z_0})_k}f$ and
\[
a_j=\frac{|\sigma_j-z_0|^2}{1-|z_0|^2}
\frac{1}{\left(1+2\Re\frac{\left(f(\sigma_j)-\sigma_j\right)\overline{z_0}}{|f(\sigma_j)-z_0|^2}\right)\beta_f(\sigma_j)-k}=\frac{1}{\beta_{\Delta_{(\mathbf{z_0})_k}f}(\sigma_j)}>0\;.
\]
Therefore we can find $h\in\Hol(\D,\C)$ with $h(\D)\subset\overline{\D}$ so that
\[
\frac{1+g(z)}{1-g(z)}-\sum_{j=1}^{n-1} a_j\frac{\sigma_j+z}{\sigma_j-z}=\frac{1+h(z)}{1-h(z)}\;.
\]
Notice that either $h(\D)\subseteq\D$ or $h\equiv e^{i\theta}\in\de\D$, and the latter case occurs if and only if we have equality in~\eqref{eq:2.CPzcin}.

If $h\equiv e^{i\theta}$ Lemma~\ref{th:1.Bz} implies that $g\in\mathcal{B}_{n-1}$. So $g$ is a rational function of degree~$n-1$; but we are assuming that the equation $g(z)=1$ has at least $n$ distinct solutions, contradiction.

So $h\in\Hol(\D,\D)$. Since $g(\sigma_n)=1$, $\beta_g(\sigma_n)=\frac{1}{a_n}$ and $g'(\sigma_n)=\overline{\sigma_n}/a_n$, where the latter equality follows from \eqref{eq:I.2.JCdue}, a quick computation yields $h(\sigma_n)=1$ and $h'(\sigma_n)=\overline{\sigma_n}/a_n$. 

Put $\tilde h(z)=zh(z)$. Then we have $\tilde h(0)=0$, $\tilde h(\sigma_n)=\sigma_n$ and $\beta_{\tilde h}(\sigma_n)=\frac{1}{a_n}+1$. So we can apply Corollary~\ref{th:1.CP} to~$\tilde h$ obtaining
\[
\Re\left(\frac{1+h(z)}{1-h(z)}\right)\ge a_n\Re\left(\frac{\sigma_n+z}{\sigma_n-z}\right)\;,
\]
with equality in one point (and hence everywhere) if and only if $h\in\Aut(\D)$. Recalling \eqref{eq:2.CPzcin} and Lemma~\ref{th:1.Bz} we see that we have proven
\eqref{eq:2.CPzqua}, with equality in one point (and hence everywhere) implying that $g$ is a Blaschke product of degree~$n$, and thus that $f$ is a
Blaschke product of degree~$n+k$, by Proposition~\ref{th:1.Blhdq}. 

In particular, \eqref{eq:2.CPzz} follows taking $z=z_0$ in~\eqref{eq:2.CPzqua}, and equality there implies that~$f\in\mathcal{B}_{n+k}$. 

To prove the converse, assume that $f\in\mathcal{B}_{n+k}$, so that $g=\Delta_{(\mathbf{z_0})_k}f\in\mathcal{B}_n$, and $\sigma_1,\ldots,\sigma_n\in\de\D$ are the $n$ distinct solutions of $g(z)=1$. Let $F\colon\C\to\widehat{\C}$ be defined by
\[
F(z)=\frac{1+g\circ\phi_{z_0}(z)}{1-g\circ\phi_{z_0}(z)}-\sum_{j=1}^n a_j\frac{\sigma_j+\phi_{z_0}(z)}{\sigma_j-\phi_{z_0}(z)}\;,
\] 
where $\phi_{z_0}(z)=(z_0-z)/(1-\overline{z_0}z)$; notice that $g\circ\phi_{z_0}$ is still a Blaschke product thanks to Lemma~\ref{th:I.2.Blaschke}.
Then $\Re F|_{\de\D}\equiv 0$; this implies $\Re F(0)=0$, which gives exactly
\[
\frac{1-|g(z_0)|^2}{|1-g(z_0)|^2}=\sum_{j=1}^n a_j\frac{1-|z_0|^2}{|\sigma_j-z_0|^2}=
\sum_{j=1}^n \frac{1}{\left(1+2\Re\frac{\left(f(\sigma_j)-\sigma_j\right)\overline{z_0}}{|f(\sigma_j)-z_0|^2}\right)\beta_f(\sigma_j)-k} \;,
\]
and we are done.
\end{proof}

\begin{Corollary}
\label{th:2.CPgen}
Let $f\in\Hol(\D,\D)$. Given $k\ge 1$, assume that $f$ is not a Blaschke product of degree at most~$k$ and that $f(0)=\cdots=f^{(k-1)}(0)=0$. 
Take $\sigma_1,\ldots,\sigma_n\in\de\D$ distinct points such that $f(\sigma_j)=\sigma_j^{k}$ and $\beta_f(\sigma_j)<+\infty$
for $j=1,\ldots,n$. Then
\[
\sum_{j=1}^{n}\frac{1}{\beta_f(\sigma_j)-k}\le\frac{1-\left|\frac{f^{(k)}(0)}{k!}\right|^2}{\left|1-\frac{f^{(k)}(0)}{k!}\right|^2}\;,
\]
with equality if and only if $f$ is a Blaschke product of degree~$n+k$.
\end{Corollary}

\begin{proof}
It immediately follows from Theorem~\ref{th:2.CPgenz} applied with $z_0=0$. 

Alternatively, we can apply directly Proposition~\ref{th:1.CPtrue} to $g=\Delta_{\mathbf{O}_{k-1}}f$. Indeed, \eqref{eq:2.Ok} shows that $g(z)=f(z)/z^{k-1}$ for $z\ne 0$
and $g(0)=0$; in particular, $g'(0)=\frac{1}{k!}f^{(k)}(0)$. Moreover, $g(\sigma_j)=\sigma_j$ and $g'(\sigma_j)=\beta_f(\sigma_j)-(k-1)$ for all $j=1,\ldots,n$; hence the assertion follows immediately from \eqref{eq:1.CPtrue}.
\end{proof}

Notice that when $k=1$ the condition \eqref{eq:2.Dksigma} becomes
\[
f(\sigma_j)=\frac{\frac{\sigma_j-z_0}{1-\overline{z_0}\sigma_j}+z_0}{1+\overline{z_0}\frac{\sigma_j-z_0}{1-\overline{z_0}\sigma_j}}
=\sigma_j\;.
\]
So \eqref{eq:2.CPzz} reduces to \eqref{eq:1.CPtrue}, and thus Theorem~\ref{th:2.CPgenz} for $k=1$ recovers exactly Proposition~\ref{th:1.CPtrue}.

\begin{Remark}
\label{rem:2.fine}
We have seen that Proposition~\ref{th:1.CPtrue} for a generic fixed point~$z_0$ followed immediately from the case $z_0=0$, just replacing the map $f$ by the composition
$\phi_{z_0}\circ f\circ\phi_{z_0}$. Such an approach however does not allow to easily deduce Theorem~\ref{th:2.CPgenz} from Corollary~\ref{th:2.CPgen} because the boundary dilation coefficient depends in a complicated way on the higher order derivatives, and so we need the iterated hyperbolic difference quotients to keep everything under control.
\end{Remark}

%Iterating the argument we get
%
%\begin{Corollary}
%\label{th:1.CPdue}
%Let $f\in\Hol(\D,\D)\setminus\Aut(\D)$ and $\sigma\in\de\D$ be such that
%\[
%\liminf_{z\to\sigma}\frac{1-|f(z)|}{1-|z|}=\beta_f<+\infty\;.
%\]
%Denote by $f(\sigma)\in\de\D$ the non-tangential limit of~$f$ at~$\sigma$. Assume that 
%\[
%f(0)=\cdots=f^{(k-1)}(0)=0\]
%for some $k\ge 1$. Then 
%\[
%\frac{\left|f(\sigma)\overline{\sigma}^k-\frac{f(z)}{z^k}\right|^2}{1-\left|\frac{f(z)}{z^k}\right|^2}\le (\beta_f-k)\frac{|\sigma-z|^2}{1-|z|^2}
%\]
%and
%\[
%\frac{\left|f(\sigma)\overline{\sigma}^k-\frac{f^{(k)}(0)}{k!}\right|^2}{1-\left|\frac{f^{(k)}(0)}{k!}\right|^2}\le\beta_f-k\;.
%\]
%Everywhere, equality holds if and only if $f$ is a Blaschke product of degree~$k+1$.
%\end{Corollary}
%
%\begin{proof}
%It suffices to apply Corollary~\ref{th:1.CP} to $g(z)=f(z)/z^{k-1}$.  It works because $g(\sigma)=\overline{\sigma}^{k-1}f(\sigma)$ and $g'(\sigma)=f(\sigma)\overline{\sigma}^k[\beta_f-(k-1)]$.
%\end{proof}

%\begin{Corollary}
%\label{th:1.CPtre}
%Let $f\in\Hol(\D,\D)$ be such that 
%\[
%f(0)=\ldots=f^{(k-1)}(0)=0
%\] 
%and $f(\sigma_j)=\sigma_j^k$ for $\sigma_1,\ldots,\sigma_n\in\de\D$. Then
%\[
%\sum_{j=1}^n \frac{1}{|f'(\sigma_j)|-k}\le\frac{1-\left|\frac{1}{k!}f^{(k)}(0)\right|^2}{\left|1-\frac{1}{k!}f^{(k)}(0)\right|^2}\;.
%\]
%Furthermore, equality holds if and only if $f$ is a Blaschke product of degree~$n+1$.
%\end{Corollary}
%
%\begin{proof}
%As in Corollary~\ref{th:1.CPtrue} using Corollary~\ref{th:1.CPdue} instead of Corollary~\ref{th:1.CP}.
%\end{proof}


\begin{bibdiv}
\begin{biblist}

\bib{Abate}{book}{
	author={M. Abate},
	title={Iteration theory of holomorphic maps on taut manifolds},
	publisher={Mediterranean Press},
	address={Rende},
	year={1989}
}

\bib{BaribeauRivardWegert}{article}{
	author={L. Baribeau and P. Rivard and E. Wegert},
	title={On hyperbolic divided differences and the Nevanlinna-Pick problem},
	journal={Comput. Methods Funct. Theory},
	volume={9},
	year={2009},
	pages={391--405}
}

\bib{Beardon}{article}{
	author={Beardon, A.F.},
	title={The Schwarz-Pick lemma for derivatives},
	journal={Proc. Amer. Math. Soc.},
	volume={125},
	year={1997}, 
	pages={3255--3256}
}

\bib{BeardonCarne}{article}{
	author={A. F. Beardon and T. K. Carne},
	title={A strengthening of the Schwarz-Pick inequality},
	journal={Amer. Math. Monthly},
	volume={99},
	year={1992},
	pages={216--217}
} 

\bib{BeardonMinda}{article}{
	author={A. F. Beardon and D. Minda},
	title={A multi-point Schwarz-Pick lemma},
	journal={J. Anal. Math.},
	volume={92},
	year={2004},
	pages={81--104}
}

\bib{Caratheodory}{article}{
	author={Carath\'eodory, C.},
	title={Untersuchungen \"uber die konformen Abbildungen von festen und ver\"anderlichen Gebieten},
	journal={Math. Ann.},
	volume={72},
	year={1912}, 
	pages={107--144}
}

\bib{Caratheodorya}{article}{,
	author={Carath\'eodory, C.},
	title={\"Uber die Winkelderivierten von beschr\"ankten analytischen Funktionen},
	journal={Sitzungsber. Preuss. Akad. Wiss. Berlin},
	year={1929},
	pages={39--54}
}

\bib{ChoKimSugawa}{article}{
	author={K. H. Cho and S.-A. Kim and T. Sugawa},
	title={On a multi-point Schwarz-Pick lemma},
	journal={Comput. Methods Funct. Theory},
	volume={12},
	year={2012},
	pages={483--499}
}

\bib{CowenPommerenke}{article}{
	author={C. C. Cowen and Ch. Pommerenke},
	title={Inequalities for the angular derivative of an analytic function in the unit disk},
	journal={J. Lond. Math. Soc.},
	volume={26},
	year={1982},
	pages={271--289}
}

\bib{Dieudonne}{article}{
	author={Dieudonn\'e, J.},
	title={Recherches sur quelques probl\`emes relatifs aux polyn\^omes et aux fonctions born\'ees d'une variable complexe},
	journal={Ann. Sci. \'Ec. Norm. Super.},
	volume={48},
	year={1931},
	pages={247--358}
}

\bib{Frovlova}{article}{
	author={A. Frovlova and M. Levenshtein and D. Shoikhet and A. Vasil'ev},
	title={Boundary distorsion estimates for holomorphic maps},
	journal={Complex Anal. Oper. Theory},
	volume={8},
	year={2014},
	pages={1129--1149}
}

\bib{Goluzin}{article}{
	author={Goluzin, G.M.},
	title={Some estimations of derivatives of bounded functions},
	journal={Mat. Sbornik},
	volume={58},
	year={1945},
	pages={295--306}
}

\bib{Julia}{article}{
	author={Julia, G.},
	title={Extension nouvelle d'un lemme de Schwarz},
	journal={Acta Math.},
	volume={42},
	year={1920},
	pages={349--355}
}

\bib{Kaptanoglu}{article}{
	author={Kaptano{\u g}lu, H.T.},
	title={Some refined Schwarz-Pick lemmas},
	journal={Michigan Math. J.},
	volume={50},
	year={2002},
	pages={649--664}
}
\bib{Komatu}{article}{
	author={Komatu, Y.},
	title={On angular derivative},
	journal={Kodai Math. Sem. Rep.},
	volume={13},
	year={1961},
	pages={167--179}
}
\bib{Mercer1997}{article}{
	author={Mercer, P.R.},
	title={Sharpened versions of the Schwarz lemma},
	journal={J. Math. Anal. Appl.},
	volume={205},
	year={1997},
	pages={508--511}
}
\bib{Mercer1999}{article}{
	author={Mercer, P.R.},
	title={On a strengthened Schwarz-Pick inequality},
	journal={J. Math. Anal. Appl.},
	volume={234},
	year={1999}, 
	pages={735--739}
}

\bib{Mercer2000}{article}{
	author={Mercer, P.R.},
	title={Another look at Julia's lemma},
	journal={Compl. Var.},
	volume={43},
	year={2000},
	pages={129--138}
}

\bib{Mercer2006}{article}{
	author={Mercer, P.R.},
	title={Schwarz-Pick-type estimates for the hyperbolic derivative},
	journal={J. Ineq. Appl.},
	year={2006},
	pages={1--6}
}

%\bib{Mercer2018}{article}{
%	author={Mercer, P.R.},
%	title={Boundary Schwarz inequalities arising from Rogosinski's lemma},
%	journal={J. Classical Anal.},
%	volume={12},
%	year={2018},
%	pages={93--97}
%}

\bib{Mercer2018}{article}{
	author={Mercer, P.R.},
	title={An improved Schwarz lemma at the boundary},
	journal={Open Math.},
	volume={16},
	year={2018},
	pages={1140-1144}
}

\bib{Osserman}{article}{
	author={Osserman, R.},
	title={A sharp Schwarz inequality on the boundary}, 
	journal={Proc. Amer. Math. Soc.},
	volume={128},
	year={2000}, 
	pages={3513--3517}
}

\bib{Picka}{article}{
	author={Pick, G.},
	title={\"Uber eine Eigenschaft der konformen Abbildung kreisf\"ormiger Bereiche},
	journal={Math. Ann.}, 
	volume={77}, 
	year={1915},
	pages={1--6}
}

\bib{Pickb}{article}{
	author={Pick, G.},
	title={\"Uber die Beschr\"ankungen analytischer Funktionen, welche durch vorgegebene Funktionswerte bewirkt werden},
	journal={Math. Ann.},
	volume={77},
	year={1915},
	pages={7--23}
}

%\bib{RenWang}{unpublished}{
%	author={G. Ren and X. Wang},
%	title={Extremal functions of boundary Schwarz lemma},
%	year={2015},
%	note={Preprint, arXiv:1502.02369}
%}

\bib{Rivard1}{article}{
	author={Rivard, P.},
	title={A Schwarz-Pick theorem for higher-order hyperbolic derivatives},
	journal={Proc. Amer. Math. Soc.},
	volume={139},
	year={2011}, 
	pages={209--217}
}
	
\bib{Rivard2}{article}{
	author={Rivard, P.},
	title={Some applications of higher-order hyperbolic derivatives},
	journal={Complex Anal. Oper. Theory},
	volume={7},
	year={2013},
	pages={1127--1156}
}

\bib{Rogosinski}{article}{
	author={Rogosinski, W.},
	title={Zum Schwarzschen Lemma},
	journal={Jahresber. Deutsch Math.-Ver.},
	volume={44},
	year={1934},
	pages={258--261}
}

\bib{Schwarz}{incollection}{
	author={Schwarz, H.A.},
	title={Zur Theorie der Abbildung},
	pages={108--132}, 
	booktitle={Gesammelte Mathematische Abhandlungen, II},
	publisher={Springer}, 
	address={Berlin}, 
	year={1890}
}

\bib{Unkelbach1}{article}{
	author={Unkelbach, H.},
	title={\"Uber die Randverzerrung bei konformer Abbildung},
	journal={Math. Z.},
	volume={43},
	year={1938},
	pages={739--742}
}

\bib{Unkelbach2}{article}{
	author={Unkelbach, H.},
	title={\"Uber die Randverzerrung bei schlichter konformer Abbildung},
	journal={Math. Z.},
	volume={46},
	year={1940},
	pages={329--336}
}

\bib{Wolff}{article}{
	author={Wolff, J.},
	title={Sur une g\'en\'eralisation d'un th\'eor\`eme de Schwarz},
	journal={C.R. Acad. Sci. Paris},
	volume={183},
	year={1926},
	pages={500--502}
}

\end{biblist}
\end{bibdiv}
\end{document}